\documentclass{amsart}
\usepackage{xcolor}
\usepackage{mathabx}
\usepackage{smartdiagram}
\usepackage{tikz}
\usepackage{hyperref}
\usepackage{amsmath,bm,bbm,amsthm, amssymb}
\usepackage{soul}
\usepackage{color}
\usepackage{dsfont}
\usepackage{animate}
\usepackage{etoolbox}
\usepackage{comment}
\usepackage{pgf}
\usetikzlibrary{calc}
\usetikzlibrary{patterns}
\usetikzlibrary{arrows}
\usetikzlibrary{decorations.pathreplacing}
\usepackage[utf8]{inputenc}
\usepackage{pgfplots}
\usepackage[most]{tcolorbox}
\usepackage{adjustbox}
\usepackage[final]{pdfpages}
\usepackage[ruled,vlined,linesnumbered,algosection,resetcount]{algorithm2e}
\usepackage{blkarray}
\usepackage{wrapfig}

\theoremstyle{plain}
\newtheorem{theorem}{Theorem}
\newtheorem{proposition}[theorem]{Proposition}
\newtheorem{corollary}[theorem]{Corollary}
\newtheorem{lemma}[theorem]{Lemma}

\theoremstyle{definition}
\newtheorem{definition}[theorem]{Definition}

\newtheorem{example}[theorem]{Example}

\theoremstyle{remark}
\newtheorem{remark}[theorem]{Remark}

\definecolor{wiasblue}   {cmyk}{1.0, 0.60, 0, 0}
\definecolor{mlugreen}{RGB}{172,6,52}
\definecolor{darmstadt}{RGB}{135,206,250}

\def\E{\mathbb E}

\def\P{\mathbb P}

\def\R{\mathbb R}
\def\T{\overline{\mathbb T}}
\def\W{\mathbb T}
\def\bbT{\mathbb T}
\def\Z{\mathbb Z}
\def\mc{\mathcal}
\def\ms{\mathsf}

\def\s{\sigma}
\def\su{\subseteq}

\def\e{\varepsilon}
\def\t{\tau}
\def\g{\gamma}
\def\b{\beta}

\def\om{\omega}

\def\one{\mathbbmss{1}}

\def\De{\Delta}

\def\Ga{\Gamma}
\def\G{\mathbb G}

\def\ff{\infty}

\def\vp{\varphi}

\def\d{{\rm d}}
\def\dif{{\rm d}}

\def\NN{\mc N}

\def\PP{\mc P}
\def\Pn{\mc P_n}

\def\C{\mathbb C}
\def\CC{\mc C}

\def\DD{D}

\def\TT{\mc T}

\def\f{\frac}

\def\ua{\ms{up}}

\def\im{\item}
\def\sm{\setminus}

\def\bep{\begin{proof}}
\def\enp{\end{proof}}
\def\bepr{\begin{proposition}}
\def\enpr{\end{proposition}}
\def\bec{\begin{corollary}}
\def\enc{\end{corollary}}
\def\bea{\begin{align}}
\newcommand\eea{\end{align}}
\def\beas{\begin{align*}}
\def\eeas{\end{align*}}
\def\bet{\begin{theorem}}
\def\ent{\end{theorem}}
\def\bee{\begin{example}}
\def\ene{\end{example}}

\def\da{\ms{do}}

\def\bede{\begin{definition}}
\def\ende{\end{definition}}
\def\ber{\begin{remark}}
\def\enr{\end{remark}}
\def\beca{\begin{cases}}
\def\enca{\end{cases}}
\def\bel{\begin{lemma}}
\def\enl{\end{lemma}}
\def\been{\begin{enumerate}}
\def\enen{\end{enumerate}}

\def\beit{\begin{itemize}}
\def\enit{\end{itemize}}
\def\befr{\begin{frame}}
\def\enfr{\end{frame}}
\def\ti{\times}

\def\Var{\ms{Var}}
\def\Cov{\ms{Cov}}

\renewcommand\le{\leqslant}
\renewcommand\ge{\geqslant}

\def\X{\mathbb X}

\def\becbb{\begin{center}\begin{tcolorbox}[{colback=Dandelion!20}]}
\def\encbb{\end{tcolorbox}\end{center}}
\def\beccb{\begin{center}\begin{tcolorbox}[{colback=Dandelion!20}]}
\def\enccb{\end{tcolorbox}\end{center}}
\def\becb{\begin{center}\begin{tcbox}[{colback=Dandelion!20}]}
\def\encb{\end{tcbox}\end{center}}

\def\bef{\begin{figure}[!h]}
\def\enf{\end{figure}}
\def\bels{\begin{lstlisting}}
\def\enls{\end{lstlisting}}

\def\betp{\begin{tikzpicture}}
\def\entp{\end{tikzpicture}}
\def\co{\colon}
\def\endo{\end{document}}

\def\RR{\mc R}

\def\N{\mathbb N}

\def\MalD{D}

\newcommand{\mP}{\mathcal P}
\newcommand{\inDt}{D_\ms{in}}
\newcommand{\inD}{D_\ms{in}}

\def\RR{\overline{\mathbb R}}

\usepackage{cancel}
\usepackage{fullpage}

\begin{document}

\title{Normal approximation for subgraph counts in age-dependent random connection models}

\author{Christian Hirsch}
\address[Christian Hirsch]{Department of Mathematics\\ Aarhus University \\  Ny Munkegade, 118, 8000, Aarhus C,  Denmark.}
\email{hirsch@math.au.dk}
\author{Takashi Owada}
\address[Takashi Owada]{Department of Statistics, Purdue University, West Lafayette, 47907, USA}
\email{owada@purdue.edu}
\author{Raphaël Lachièze-Rey}
\address[Raphaël Lachièze-Rey]{INRIA \\ Paris, France.}
\email{raphael.lachieze-rey@math.cnrs.fr}

\begin{abstract}
	In this paper, we study the topic of asymptotic normality of certain subgraph counts in a model of spatial scale-free random networks known as \emph{age-dependent random connection model}. More precisely, in the light-tailed regime where the subtree counts centered at a typical vertex have finite second moments, we prove the asymptotic normality of clique- and subtree-counts {when the limit is not stable}. In the case of the clique-counts, we establish a multivariate quantitative convergence through the Malliavin-Stein method. In the more delicate case of subtree-counts, we obtain the distributional convergence to a normal distribution by using a classical central limit theorem for sequences of associated random variables.
\end{abstract}
\maketitle

\section{Introduction}
\label{sec:intro}
%
%
Recently, several variants of kernel-based or weighted \emph{random connection models (RCMs)} have received substantial research interest \cite{glm2,komjathy2,glm,komjathy}. This is because they can reproduce the heavy-tailed degree distributions observed in many real-world complex networks, while simultaneously retaining a positive clustering coefficient. This latter property sets  them apart from standard discrete models for complex networks such as the Chung-Lu or the Norros-Reittu model \cite{ChungLu,NorrosReittu}. Other popular alternatives for obtaining a positive clustering coefficient include the spatial preferential attachment models 
\cite{jacMor1,jacMor2}. While these networks offer a natural interpretation of a growing network, their time-dependent construction makes them far less tractable than the weighted RCMs. Another highly studied class of spatial scale-free networks is the hyperbolic random geometric graph from \cite{gpp}{, where the non-Euclidean geometry
can pose technical difficulties. Such hyperbolic networks actually have Euclidean RCM counterparts, in the sense that these networks have the same local limits as some RCMs (\cite[Theorem 9.33]{book-hof}),
but it is not clear if local limits are sufficient to capture many interesting network properties such as the asymptotic normality of subgraph counts.}

%
%
In this paper, we consider a popular example for such weighted RCMs, namely the \emph{age-dependent random connection model (ADRCM)} \cite{glm2,komjathy2,glm,komjathy}. Here, nodes are placed in a metric space and endowed with i.i.d.~weights, and the connection probability between two nodes is a function of both the distance and their respective weights{, in such a way that the typical degree has a power law}. The main achievements of our paper are normal approximation results for  clique and subtree counts in the ADRCM. 
Our motivation is to investigate the limiting behavior of such networks in large domains, and such normal approximation results would statistically validate the use of the ADRCM as a model for real data. For example, with the help of the obtained asymptotic normality, we could derive asymptotic confidence intervals for such networks as the basis for various hypothesis tests. Even though computations and regimes may not be identical to those for other RCMs, many of the methods employed here have a potential to be deployed in other models, where the local number of subgraphs is square integrable.

Central limit theorems (CLTs) in spatial random systems have become a vibrant research stream in the last twenty years, from the martingale-based techniques of  \cite{barysh,yukCLT,Yu15} to the Malliavin-Stein method which allowed to obtain second-order Poincaré inequalities and likely optimal rates of convergence \cite{mehler,ustat,y2,y1,LrP17}  for related works in this direction. However, these results traditionnally require uniformly bounded fourth moments, which  is not compatible with the heavy-tailed behavior of the ADRCM. We exploit here a  recent update of the Malliavin-Stein method from \cite{trauth,trauth2} that allows  to deal with functionals where only moments of order $(2+\varepsilon)$ are needed, and  \cite[Theorems 4.1 and 4.8]{oliveira:2012} for associated random variables.

The main contributions of this paper are as follows:
\begin{enumerate}
	\item For clique counts, we derive the limiting variance and prove   quantitative normal approximation using the Malliavin-Stein bounds from \cite{trauth, trauth2}. This analysis allows us to study the multivariate case as well, where several clique counts of various dimensions are considered simultaneously.
	\item For subtree counts, we give the variance magnitude and show the qualitative normal approximation based on the CLTs for associated random variables from Theorems 4.1 and 4.8 in \cite{oliveira:2012}.
\end{enumerate}
One of the most challenging issue is to deal with the lack of integer powers of the quantities of interest, hence we  carry out a detailed configurational analysis of both clique and subtree counts, in order  to deduce the required fractional moment bounds. Additionally,  we  need to carry out their detailed second-order analysis.

In the ADRCM, the subgraph counting statistics follow a wide range of limiting behaviors from Gaussian to stable laws. In our companion paper  \cite{ht24}, we established stable limit theorems for both clique and subtree counts in the regime where these statistics have \emph{infinite}  second moments. In contrast, the present work focuses on proving normal approximation results for the same statistics under parameter regimes     of  \emph{finite} second moments. By combining the results of \cite{ht24} with those of the present work, we obtain a full picture of a phase transition from Gaussian to stable laws.   
This phase transition depends   critically on the parameter regime of the ADRCM,  and such subtle dependence becomes even more pronounced for  general subgraph counts. Indeed, unlike classical random geometric graphs (see, e.g., the monograph~\cite{penrose}),   the situation for the ADRCM is far more delicate due to the heavy-tailed phenomena induced by the associated weights. As a result, a unified treatment of CLTs, as well as stable limit theorems, for arbitrary subgraph counts, still  remains   out of reach. 

Nevertheless, the results of this paper  mark a   crucial step toward understanding the behavior of fundamental statistics in network analysis. One such example is the \emph{clustering coefficient} considered in \cite{jacMor1, pim}. This index measures how well a model resembles  an actual network, and is typically   defined as the ratio between wedge counts (a special case of subtree counts) and triangle counts (a special case of clique counts). It is known from \cite{jacMor1, pim} that the clustering coefficient in the ADRCM is  asymptotically positive. 
By combining \cite{ht24} and the current work, we are expected to develop a unified theory for clustering coefficients, regardless of whether the underlying statistics have finite or infinite second moments.  For example, in the light-tailed regime with finite second moments, we conjecture that    the normalized clustering coefficient may  be approximated by the ratio of components of a Gaussian vector, 
taking the form $a + b Z_1/Z_2$ (Cauchy distribution), where $Z_1$ and $Z_2$ are i.i.d.~normal random variables.  Further investigation of the clustering coefficient is left for future research.

The rest of the manuscript is organized as follows. In Section \ref{sec:model}, we introduce the model and state our main results. To provide the reader with a broad overview,  proof outlines are given in Section \ref{sec:outline}. A key component in the clique case will be variance and covariance asymptotics, which  are presented in Section \ref{sec:mom}. The remaining arguments for the clique CLT are provided in Section \ref{sec:cliq}. Finally, the subtree case is treated in Section \ref{sec:tree}. We conclude with the Appendix containing  more technical computations.

\section{Model and main results}
\label{sec:model}

We review  the ADRCM from \cite{glm2,glm}. Let $n\ge 1$, 
$0< \g < 1$, and $ \PP_n =(P_i)_{i \ge 1} = \big((X_i,U_i)\big)_{i\ge1} \su \T_n :=\W_n \ti [0, 1]$ be a unit-intensity Poisson point process, where $\W_n := [-n/2,n/2]^d/\sim$ represents the $d$-dimensional torus.  Given $P=(X,U)\in \PP_n$, 
we interpret $X$ as the spatial (or location) coordinate and $U$ as the associated {birth time}, or \emph{ mark},  of $X$. 

We connect two nodes $P = (X, U)$ and $Q = (Y, V)$ with  $V > U$, in symbols $Q \to P$, with probability  
$$\P(Q\to P)=\varphi(|X - Y|^d  U^\g V^{1-\g}/ \b),$$ $\b >0$,
where $|\cdot |$ is the Euclidean norm and  $\varphi:\R_+\to \R_+$ is a  non-increasing and integrable profile function. 
To ease our computations,  we  always consider the  case $d = 1$ and the hard profile function $\vp(r) = \one\{r\le 1\}$, and sometimes $\b=1$.
The resulting graph is known as the \emph{age-dependent random connection model} (ADRCM) and is denoted by $\G_{n, \g, \b}$. 
Sometimes, we think of $\G_{n, \g, \b}$ as a directed graph by  interpreting an edge between $P=(X,U)$ and $Q=(Y,V)$ with $V > U$ as being directed from the younger node $Q$ to the older node $P$.

\subsection{CLT for clique counts}
We begin with introducing some notation. For $k\ge1$,  let 
\begin{equation}  \label{e:def.clique.counts}
\CC_{n, k}:= \big|\{\De \su \PP_n\co |\De| = k, \, p\leftrightarrow q \text{ for all $p, q \in \De$} \big\}\big|,
\end{equation}
be the number of $k$-cliques where $ \leftrightarrow$ means   the point with  higher mark (i.e., larger time coordinate) is connected to the  point of  lower mark (i.e., smaller time coordinate).  In the regime $0 < \g < 1/2$, 
 {a point $(x,u) \in \widebar \bbT_n$ is less likely to connect to   other nodes of higher marks,  even when $u$ is small.    Consequently, a CLT holds for \eqref{e:def.clique.counts}, as shown in Theorem~\ref{thm:cliq} below. 
In contrast, when $\g$ is larger, i.e., $1/2 < \g < 1$, the number of connections to $(x,u) \in \widebar \bbT_n$ increases significantly, even for small $u$, which yields a stable limit theorem for \eqref{e:def.clique.counts}, as proved in \cite{ht24}.

Recall that the Wasserstein distance between two integrable random variables $X,Y$ is defined as 
$$
d_{\ms W}(X,Y)=\sup_{h\in \text{Lip}_1}\big| \E [h(X)]-\E [h(Y)]\big|,$$
where $\text{Lip}_1$ is the space of $1$-Lipschitz functions $h:\mathbb R\to \mathbb R.$ 
%
%
\bet
\label{thm:cliq}
 Let $0<\g < 1/2$, $k_0 \ge 1$ and $\eta \in (1,2)$ be such that $\eta (2\g \vee ( 1-\g)) < 1$. 
\been 
\im{\bf Univariate CLT.} 
For every $1\le k \le k_0$, it holds that 
\begin{equation}  \label{e:d_W.conv}
	d_{\ms W}\Big(\widetilde{\CC_{n,k}}, \NN(0, 1)\Big)   \in O\big(n^{-(\eta - 1)/2}\big),
\end{equation}
where $\widetilde{\CC_{n, k}}:=(\CC_{n, k} - \E[\CC_{n, k}])/{\sqrt{\Var(\CC_{n, k})}}$, $\mathcal N$ is a normal distribution. Moreover, for some $\sigma_k\in (0,\infty)$, we have 
$$
\big|  n^{-1}\Var(\CC_{n, k}) - \sigma_{k} \big| \in O(n^{2\g-1}). 
$$
\im{\bf Multivariate CLT.} There exists a semi-definite positive matrix $\Sigma=(\sigma_{k,\ell})$, such that 
\begin{equation}  \label{e:d_3.conv}
	d_3\Big(\big( \widetilde{\CC_{n,k}}\big)_{k \le k_0}, \NN(0, \Sigma)\Big)   \in O\big(n^{-(\zeta(\g) \wedge ((\eta - 1)/2))}\big),
\end{equation}
where   $d_3$ is the distance function on random vectors based on $\mathcal C^3$-test functions  defined  in \cite[Equation (2.3)]{trauth2} and $\zeta(\g) := (1- \g) \wedge (1-2\g)/\g$. Moreover, $\sigma_{k,\ell}$ is explicitly defined later in Proposition \ref{pr:lvarcliq}, and satisfies $\sigma_{k, k}>0$. For every $k, \ell \in \{ 1,\dots,k_0 \}$, it holds that 
$$\big|  n^{-1} \Cov(\CC_{n, k}, \CC_{n, \ell}) - \sigma_{k, \ell} \big| \in O(n^{-\zeta(\g)}). $$
In particular, \eqref{e:d_3.conv} implies that $\big( \widetilde{\CC_{n, k}}\big)_{k \le k_0} \xrightarrow{d} \NN(0, \Sigma)$ as $n\to\infty$, where $\xrightarrow{d}$ is convergence in law.  
\enen
\ent
In Theorem \ref{thm:cliq} we do not  attempt to optimize the rate of convergence because    our goal is to develop arguments  applicable for all $0<\g < 1/2$. We believe, however, that a more refined analysis, potentially yielding sharper rates, is feasible in the more restrictive regime $0 < \g < 1/4$.

%
%
\subsection{CLT for subtree counts}
Next, we deal with  subtree counts, where we often identify a directed tree with its edge set. While an (unoriented) clique is parameterized fully by the number of nodes only, the oriented nature of the underlying tree offers a wide variety of different configurations which lead to different asymptotic behaviours. We henceforth characterize different types of trees by fixing a  vertex set $\{1, \dots, m\}$ and embedding a directed  tree structure $\ms T$ on this vertex set. Then, 
\begin{equation}  \label{e:subtree.count}
\TT_{n, \ms T} := \big|\big\{\{P_1, \dots, P_m\} \su \PP_n\co P_i\to  P_j \text{ for all edges } \{i,j\} \text{ in } \ms{T}  \text{ with } i \to j\big\}\big|
\end{equation}
is the number of copies of $\ms T$. More precisely, it represents the number of injective homomorphisms from an abstract directed tree $\ms{T}$ into $\G_{n, \g, \b}$. 
Since the connectivity patterns of general trees are more delicate than in the clique case, the analysis required to obtain various quantitative bounds becomes significantly more complex. For this reason, we do not attempt to deduce an explicit rate of convergence.

For us, it is convenient to reformulate \eqref{e:subtree.count} as follows. First, let us  endow $ \ms T$ with a root $ \ms r$, and for $(x,u) \in \widebar \bbT_n$, define $\inDt(x,u)$ as the number of {copies of $ \ms T$} such that the root $\ms{r}$ of $\ms{T}$ is mapped to $(x,u)$. We write   $\inDt(u) := \inDt(0,u)$. 
Now, using this notation, \eqref{e:subtree.count} can  be rewritten as 
\begin{equation} \label{e:def.sub-tree.counts}
	\mc T_{n, \ms T} = \sum_{P \in \mP_n} \inDt(P),
\end{equation}
where the roots   of subtrees in $\G_{n, \gamma, \beta}$ are drawn from $\Pn$. It is important to   emphasize that $\inDt(P)$ enumerates graph homomorphisms; hence, \eqref{e:def.sub-tree.counts} does not necessarily count \emph{induced} subtrees in $\G_{n,\g,\b}$. Figure \ref{fig:T}    gives a simple example of wedge counts.

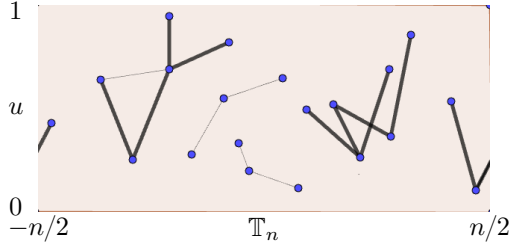
\begin{figure}[h!] 

\begin{center}
    \begin{tikzpicture}
        \node[anchor=center] (img) at (0,0) {\includegraphics[width=6cm]{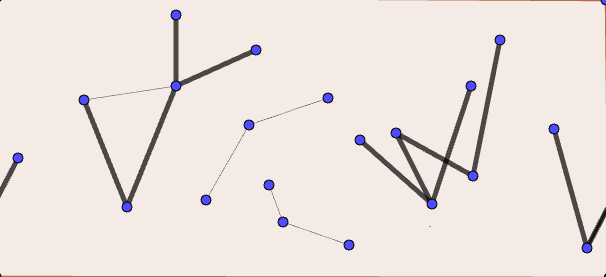}};


        \node at (-3.3, 1.3) { $1$};
        \node at (-3.3, 0) { $u$};
        \node at (-3.3, -1.3) {$0$};

        \node at (0,-1.6) {$\mathbb{T}_n$};
        \node at (-3, -1.6) {$-n/2$}; 
        \node at (3, -1.6) {$n/2$};

    \end{tikzpicture}
\caption{Example of directed wedge   trees of two leaves (i.e., $\ell=2$), where two vertices are connected to a root with a   lower mark. The     quantity $\mathcal T_{n,\ms{T}}$ counts only the ``visualized'' rooted wedges   drawn by solid lines. Note that     each such rooted wedge admits two isomorphisms, and therefore we have $\mathcal T_{n,\ms{T}} = 14$. }
 \label{fig:T}

\end{center}

 \end{figure}

Our second main result is a CLT for the subtree counts \eqref{e:def.sub-tree.counts} in the regime $0 < \g < 1/(2\ell)$, where $\ell$ denotes the number of leaves in $\ms T$ (with no further children). Notice that the  threshold value $1/(2\ell)$ is optimal in the sense that stable limit theorems for \eqref{e:def.sub-tree.counts} have been established when $1/(2\ell) < \g < 1$ (see \cite{ht24}).

%
\begin{theorem} \label{t:CLT.sub-tree}
Given  a directed tree $\ms{T}$, assume that $0< \gamma < 1/(2\ell)$. Then, we have  {$c_-n \leq \text{Var}(\mc T_{n, \ms T})\leq c_+n$} for some $0<c_-<c_+<\infty$,  and
$$\f{\mc T_{n, \ms T}-\E[\mc T_{n, \ms T}]}{\sqrt{\ms{Var}(\mc T_{n, \ms T})}} \Rightarrow \mathcal N(0,1), \ \ \ n\to\infty. $$
\end{theorem}

\begin{remark}[Higher dimensions and non-compactly supported profile functions]
	\label{rem:dim}
While Theorems \ref{thm:cliq} and \ref{t:CLT.sub-tree} are proved only  in dimension $d=1$, there is no principal obstruction to a generalization to higher-dimensions. However, since the one-dimensional setting already relies on heavy computations, carrying out those calculations in higher-dimensions would make our analyses far less accessible without gaining much additional mathematical insights. 
Moreover, extending our results to profile functions $\vp$ of unbounded  support would be substantially more involved. This is because for compactly supported profile functions, the size of joint neighborhood of nodes can be accurately bounded via a simple application of the triangle inequality. 
\end{remark}

%
%
\section{Proof outline}
\label{sec:outline}
In this section, we give the outlines for the proofs of Theorems \ref{thm:cliq} and \ref{t:CLT.sub-tree}. These two results are shown using different techniques, so we treat them in Sections \ref{ssec:cliq} and \ref{ssec:tree}, separately.

%
%
\subsection{Proof outline for Theorem \ref{thm:cliq}}
\label{ssec:cliq}
First, we give a broad overview of the proof of Theorem \ref{thm:cliq}; the detailed arguments are provided in the subsequent proof sections.  
Our derivation of the multivariate quantitative   CLTs in \eqref{e:d_3.conv} is based on \cite[Theorem 1]{trauth2}, while the univariate version in \eqref{e:d_W.conv} uses \cite[Theorem 3.4]{trauth}.  Both results are suitably adapted to handle random variables whose fourth     moments are not uniformly bounded.

%
%
We begin with recalling the general setup. For a general Poisson functional $F:\X\to \mathbb R$ on a state space $\X$, we introduce the first- and second-order Malliavin derivatives: for $q\in \X$ with $\P(q\in\PP)=0$,  
$$
D_q F(\PP):=F(\PP\cup\{q\})-F(\PP),
$$
and for $p,q\in\X$,
$$D_{p,q}F(\PP):=D_p\big(D_q F(\PP)\big)=F(\PP\cup\{p,q\})-F(\PP\cup\{p\})-F(\PP\cup\{q\})+F(\PP),$$
which is  obviously symmetric in $p$ and $q$. 
In our case, the Poisson process is $\mP_n$, with unit intensity on  the state space $\X=\T_n$ with respect to Lebesgue measure, and the functional $F$ is chosen as the clique count, i.e., $F (\mP_n):= \CC_{n, k}$.  For simplicity, we write $\MalD_u := \MalD_{(0, u)}$ and  $\MalD_{u, q} := \MalD_{(0, u), q}$ for $u\in (0,1)$ and $q\in \widebar \bbT_n$.
In particular, note that 
\been
\im $D_q \CC_{n,k} $ (resp.~$D_u \CC_{n,k}$) is the number of $k$-cliques containing $q$ (resp.~$(0,u)$), 
\im $D_{u,q}\CC_{n,k} $ is the number of   $k$-cliques  containing $(0,u)$ and $q$. 
\enen
For a fixed $\eta \in (1, 2)$, we need to prove  that the following quantities are  of at most volume order $O(|\T_n|)=O(n)$: for $p, q\in \widebar \bbT_n$, 
%
%

\begin{align*}
	&\Ga_1 :=   \sum_{k,\ell\le k_0}\int_{\T_n} \left(\int_{\T_n} \E\left[(\MalD_q\CC_{n, k})^{2\eta}\right]^{1/({2\eta})}  \E\left[(\MalD_{p,q} \CC_{n, \ell})^{2\eta}\right]^{1/({2\eta})} \d q \right)^\eta \d p, \\
	&\Ga_2 :=   \sum_{k,\ell\le k_0} \int_{\T_n} \left(\int_{\T_n} \E\left[(\MalD_{p,q} \CC_{n, k})^{2\eta} \right]^{1/{2\eta}}\E\left[(\MalD_{p,q} \CC_{n, \ell})^{(2\eta)} \right]^{1/{(2\eta)}} \d q \right)^\eta \d p, \\
	&\Ga_3 :=  \sum_{k,\ell\le k_0}\int_{\T_n} \E\big[\left(\MalD_p\CC_{n, k} \right)^{\eta + 1}\big]^{1/(\eta + 1)} \E\big[\left(\MalD_p\CC_{n, \ell} \right)^{\eta + 1}\big]^{1-1/(\eta + 1)}\d  p.
\end{align*} By the invariance under torus translations, these quantities can be simplified as $\Ga_i = |\T_n| \Ga_i'=n \Ga_i'$, where
\begin{align}
	 \label{Ga1}\Ga_1' &:=   \sum_{k,\ell\le k_0}\int_{0}^1 \left(\int_{\T_n} \E\left[(\MalD_q\CC_{n, k})^{2\eta}\right]^{1/(2\eta)}  \E\left[(\MalD_{u,q} \CC_{n, \ell})^{2\eta}\right]^{1/(2\eta)} \d q \right)^\eta \d u, \\
	\label{Ga2}\Ga_2' &:=   \sum_{k,\ell\le k_0} \int_0^1 \left(\int_{\T_n} \E\left[(\MalD_{u,q} \CC_{n, k})^{2\eta} \right]^{1/{2\eta}}\E\left[(\MalD_{u,q} \CC_{n, \ell})^{(2\eta)} \right]^{1/{(2\eta)}} \d q \right)^\eta \d u, \\
	\label{Ga3}\Ga_3' &:=  \sum_{k,\ell\le k_0}\int_0^1 \E\big[\left(\MalD_{ u}\CC_{n, k} \right)^{\eta  + 1}\big]^{1/(\eta  + 1)} \E\big[\left(\MalD_{u}\CC_{n, \ell} \right)^{\eta + 1 }\big]^{1-1/(\eta + 1)}\d u.
\end{align}
After taking the exponent $1/\eta$ in \eqref{Ga1}--\eqref{Ga3}, we can completely recover the quantities in \cite[Theorem 1]{trauth2}.

To bound them, we must  control $(2 + \e)$th moment of the first- and second-order difference operators for small $\e>0$. 
In addition to controlling the terms $\Ga_i'$, $i=1,2,3$,
 we also need to show that the variance of clique counts grows at least linearly in the system size.  Moreover, for the multivariate setting, it is necessary to determine the exact growth rate of covariances as well. 

To make our approach more precise, let $\PP$ be a  Poisson point process on $\RR := \R \times [0, 1]$ {with Lebesgue intensity}. Then,  given  $k, \ell \ge 1$, $p, q \in \RR$, and the ADRCM defined  on $\PP\cup\{p\}$, we let $\CC_k(p)$
 be the number of $k$-cliques \emph{centered} at $p$, by which we mean the number of $k$-cliques with $p$ as its lowest marked vertex. On $\PP \cup \{p, q\}$, let  $\CC_{k, \ell}(p, q)$ denote the number of point configurations giving rise to a $k$-clique centered at $p$ and an $\ell$-clique centered at $q$, such that these two cliques have at least one common point. A common point could  be $p$ or $q$ themselves. To ease notation, for $u \in (0, 1)$ and $q\in \RR$,  we put $\CC_k(u):=\CC_k((0,u))$ and $\CC_{k, \ell}(u, q):= \CC_{k, \ell}((0, u), q)$.

%
%
\bepr[Covariance asymptotics] 
\label{pr:lvarcliq}
Let $0<\g < 1/2$ and $k,\ell\in \{ 1,\dots,k_0 \}$. Then, 
$$
\big| n^{-1} \ms{Cov}(\CC_{n, k}, \CC_{n, \ell}) - \sigma_{k,\ell} \big| \in O(n^{-\zeta(\g)}), 
$$
where 
	$\s_{k, \ell} := \int_0^1 \E\big[\CC_{k}(u) \CC_{\ell}(u)\big] \d u +	\int_0^1 \int_{\RR}\E\big[\CC_{k, \ell}(u, q)\big] \d q\d u\in [0,\infty). 
	$ {It is in particular easily established that $\sigma_{k,k}>0$.}
\enpr

%
%
\bepr[Bounds for $\Ga_i$]
\label{pr:gisu}
Let $0<\g < 1/2$ and $\eta \in (1,2)$ be such that $\eta (2\g \vee (1 - \g))< 1$. Then $\Ga_1' \in O((\log n)^\eta)$ and
$ \Ga_2' + \Ga_3'\in O(1)$.
\enpr
Once these results are established,  Theorem \ref{thm:cliq} follows directly from \cite[Theorem 3.4]{trauth} and  \cite[Theorem 1]{trauth2}. {In the univariate case, the exponent is optimal using this strategy as we cannot make better than proving that the $\Gamma'_i$'s are bounded in Proposition \ref{pr:gisu}.}

{
\begin{remark}[Kolmogorov distance] For statistical applications, it is sometimes relevant to use Kolmogorov's distance between two random variables $$d_{\ms K}(X,Y)=\sup_{t\in \mathbb{R}}|\P(X\leq t)-\P(Y\leq t)|.$$
In the univariate case, \cite[Theorem 3.4]{trauth} also gives bounds on $d_{\ms K}\Big(\widetilde{\CC_{n,k}}, \NN(0, 1)\Big)$ in terms of quantities $\gamma_4,...,\gamma_7$, which are related to estimates on first- and second-order derivatives of the functionals of interest (quantities $\gamma_1,\gamma_2,\gamma_3$ bound Wasserstein distance). At the cost of some additional computations, it is likely that we could obtain  Kolmogorov bounds using this route. 
\end{remark}}

 %
 %
 \subsection{Proof outline for Theorem \ref{t:CLT.sub-tree}}
 \label{ssec:tree} 
Theorem \ref{t:CLT.sub-tree} is proved using the CLTs for associated random variables from Theorems 4.1 and 4.8 in \cite{oliveira:2012}. We now provide an outline of the proof.
To begin, we decompose $\mc T_{n, T}$ in \eqref{e:def.sub-tree.counts} based on the spatial coordinates of $P \in \Pn$: 
$$
\mc T_{n, T} = \sum_{i=1}^n \sum_{P \in \mP_n \cap ([i-1,i) \times [0,1])} \inDt(P) =: \sum_{i=1}^n T_i^{(n)}. 
$$
Due to the cyclic structure of the torus, the random variables $(T_i^{(n)})_{i=1}^n$ are identically distributed though not independent. Moreover,  $(T_i^{(n)})_{i=1}^n$ is \emph{positively associated}, which means that for every $k \ge 1$ and any two {bounded}  functions $f_1, f_2: \R^k \to \R$ that are non-decreasing in the coordinatewise sense, we have 
\begin{equation}  \label{e:pos.cov}
\ms{Cov} \big( f_1(T_1^{(n)}, \dots, T_k^{(n)}), f_2(T_1^{(n)}, \dots, T_k^{(n)}) \big) \ge 0. 
\end{equation}
To see \eqref{e:pos.cov}, let $\Omega$ denote the space of finite subsets of $\T_n$, and define a partial ordering $\om \le \om'$ for $\om, \om' \in \Omega$ if and only if $\om \subset \om'$. Adding points to $\T_n$ naturally increases the chance of forming (not necessarily induced) subtrees. Thus, the Harris-FKG inequality \cite[Theorem 17.4]{poisBook} shows the required
$$
\E \big[ f_1(T_1^{(n)}, \dots, T_k^{(n)}) f_2(T_1^{(n)}, \dots, T_k^{(n)}) \big] \geq \E \big[ f_1(T_1^{(n)}, \dots, T_k^{(n)}) \big] \E \big[ f_2(T_1^{(n)}, \dots, T_k^{(n)}) \big].
$$
Note that the centered versions $\bar T_i^{(n)}:= T_i^{(n)}-\E \big[ T_i^{(n)} \big]$, $i=1,\dots, n$  are also  positively associated.

Next, we define the so-called \textit{Cox-Grimmett coefficient} (see \cite{cox:grimmett:1984}): for each $1\le k \le n$, 
\begin{equation} \label{e:CG.coeff}
u_n(k) := \max_{1\le p \le n} \sum_{j=1, \, |j-p|\ge k}^n \ms{Cov} (T_p^{(n)}, \, T_j^{(n)}). 
\end{equation}
Due to the cyclic structure of  torus, \eqref{e:CG.coeff} simplifies as 
$
u_n(k)  = 2\sum_{j=k+1}^{\lceil n/2\rceil} \ms{Cov} (T_1^{(n)}, T_j^{(n)}). 
$
A key feature  for the proof of Theorem \ref{t:CLT.sub-tree}  is the uniform boundedness of $u_n(k)$. To this end, Lemma \ref{l:cov.var.bdd.tree} below  plays an important role. In fact, it follows from Lemma \ref{l:cov.var.bdd.tree} $(i)$, together with the assumption $0< \gamma < 1/(2\ell)$, that
\begin{equation}  \label{e:u.nk.bdd}
 \sup_{n} u_n(k) \le C  \sum_{j=k+1}^\infty j^{-(2-2\ell\gamma)} \le C  k^{-(1-2\ell\gamma)}. 
\end{equation}
\begin{lemma}  \label{l:cov.var.bdd.tree}
	Let $\ms T$ be a directed tree with $\ell \ge 1$ leaves. 
	Then,  it holds that $(i)$
	$\sup_{n \ge 1}\ms{Cov} (T_1^{(n)}, \, T_k^{(n)}) \in O( k^{-(2-2\ell\gamma)})$ and $(ii)$
	$\inf_{n \ge 1}{\Var}(\mc T_{n, T})/n>0$.
\end{lemma}
The uniform boundedness in \eqref{e:u.nk.bdd}, along with the positive associativity of  $(T_i)_{i=1}^n$, enables us to complete the proof by following an approach similar  to that of Theorems 4.1 and~4.8 in \cite{oliveira:2012}.

 %
 %
 \section{Mean and covariance asymptotics for clique counts}
 \label{sec:mom}
In this section, we prove the covariance asymptotics from Proposition \ref{pr:lvarcliq}. Related covariance computations were carried out  in \cite{ht24}. However, we cannot directly cite those results, as the analysis in \cite{ht24} is conducted in $\R$ with infinite volume, whereas the present work considers a torus of finite volume for each fixed $n \ge 1$, and the contribution of the boundary is actually sensitive to our analyses. Owing to this difference,    it is necessary to carefully examine the rate of convergence of the covariance in   the torus setting toward its counterpart defined on $\R$. This analysis requires substantial efforts and  constitutes  the main focus  of this  section.

We recall from \cite[Proposition 4.1]{glm} a basic result on the size of the up- and down-neighborhoods of a point in the ADRCM on $\mP$. We write 
$N^{\ua}(p) := \{q \in \RR: q\to p\}$
and
$N^{\da}(p) := \{q \in \RR: p \to q\}$
for the up- and down-neighborhoods of $p \in \RR$. Write also $N^{\ua}(u) :=  N^{\ua} ((0,u))$ for $u\in (0,1)$, and define $N^{\da}(u)$ analogously. 

%
%
\bel[Size of up- and down-neighborhoods]
\label{lem:down}	
Let $u\in (0,1)$. Then, $\PP(N^{\ua}(u))$ and $\PP(N^{\da}(u))$  are independent Poisson variables with respective parameters $$\lambda_u^+=c_+ u^{-\gamma},\;\lambda_u^-=c_-$$ for some $c_-,c_+>0$ depending only on $\b$ and $\g$.

In particular, the distribution of the number of down-neighbors does not depend on $u$, and if $U$ is uniformly distributed in $(0,1),$ the number of neighbors (or equivalently,  up-neighbors) of $U$ has a finite moment of order $q\in \mathbb{N}$ iff $q\g<1.$

\enl

Before proceeding, we   define the clique  counts analogously to the paragraph preceding   Proposition \ref{pr:lvarcliq}. Specifically,     given $\G_{n,\g,\b}$ (i.e., the ADRCM defined on $\mP_n$), together with $k, \ell \ge 1$ and $p, q \in \widebar \bbT_n$, let $\CC_{n,k}(p)$ be    the number of $k$-cliques    centered at $p$. Note that $\CC_{n, k}$ is the sum of the $\CC_{n, k}(p)$ for $p \in \mP_n$.  Let $\CC_{n,k,\ell}(p,q)$ be  the number of point sets    that  generate both a $k$-clique centered at $p$ and an $\ell$-clique centered at $q$, such that the two cliques  share at least one common vertex.  
As before, we also  define $\CC_{n,k}(u) := \CC_{n,k}((0,u))$ and $\CC_{n,k,\ell}(u,q) := \CC_{n,k,\ell}((0,u),q)$ for $u \in (0,1)$ and $q \in \widebar \bbT_n$.

Now, we prove Proposition \ref{pr:lvarcliq} by expanding the covariance using the Mecke formula:  

\begin{align}   
\Cov(\CC_{n, k}, \CC_{n, \ell}) &= \int_{\T_n}\E\big[\CC_{n, k}(p) \CC_{n, \ell}(p)\big]  \d p
+ \iint_{\T_n^2}\E\big[\CC_{n, k, \ell}(p, q)\big] \d (p, q)  \label{eq:ccn} \\
&= n \int_0^1\Big( \E[\CC_{n, k}(u)]\E[\CC_{n, \ell}(u)] + \Cov\big(\CC_{n, k}(u), \CC_{n,\ell}(u)\big) +	 \int_{\T_{n}}\E\big[\CC_{n, k, \ell}(u, q)\big] \d q \Big)\d u, \notag 
\end{align}  
where the second equality follows from the homogeneity of torus. 
In the remainder of the proof, our goal is to establish  the rate of convergence for each of the individual terms in \eqref{eq:ccn}. For the precise statements, however, we need to  ``redefine" all clique counts in \eqref{eq:ccn}, along with  their counterparts induced by  $\mP$,  on a common probability space. For this purpose,   we use a coupling of $\PP_{n}$ and $\PP$ by reformulating them as $\mP_n = \mP \cap \bbT_n$, where $\bbT_n = [-n/2, n/2]$. In the preceding   section, $\bbT_n$ denoted the torus; however, in this section, we define it as the interval above, with some abuse of notation. In this new formulation, we continue to define $\widebar \bbT_n := \bbT_n \times [0,1]$ as before.

%
%
\bel[Auxiliary expectation computations]
\label{lem:exp} 
Let $k\ge1$. Then, $(i)$  $\sup_{n \ge 1}\E[\CC_{n, k}(u)] \in O( u^{-\g})$ and $(ii)$ $\int_0^1 u^{-\g}\E\big[|\CC_k(u) - \CC_{n, k}(u)|\big] \d u \in O\big(n^{-\zeta(\g)}\big)$.
\enl

Next, we deal with the single-integral covariance from Lemma \ref{lem:cov1} below. The key idea is that by expanding the covariance, we need to count   pairs of suitably selected $k$-cliques   and $\ell$-cliques,  both centered at $(0, u)$, such that the two cliques share  at least one common vertex. 

%
%
\bel[Auxiliary covariance computations I]
\label{lem:cov1} 
Let $k, \ell \ge 1$. Then, 
$$		\int_0^1\big|\Cov\big(\CC_k(u) - \CC_{n, k}(u), \CC_{n, \ell}(u)\big)\big|\d u \in O(n^{-\zeta(\g)}).$$
\enl

Finally, we deal with the most complicated case, namely the double-integral contributions.

\bel[Auxiliary covariance computations II]
\label{lem:cov2} 
Let $k, \ell \ge 1$. Then,    
\been
\im[$(i)$]  $		\int_0^1\int_{\RR\sm \T_n}\E\big[\CC_{k, \ell}(u, q)\big] \d q\d u \in O(n^{-(1 - \g)})$, and  
\im[$(ii)$]	$\int_0^1\int_{\T_{n}}\E\big[\big|\CC_{k, \ell}(u, q) - \CC_{n, k, \ell}(u, q)\big|\big]  \d q \d u\in O( n^{-(1-\g)})$.
\enen
\enl

These three auxiliary results imply Proposition \ref{pr:lvarcliq} as follows. 

%
%
\bep[Proof of Proposition \ref{pr:lvarcliq}]
For the convergence of the covariance, we consider the elements in the decomposition \eqref{eq:ccn} separately. First, for the product of expected values, we apply Lemma \ref{lem:exp}. Second, for the convergence of the covariance, we apply Lemma \ref{lem:cov1}. Finally, for the convergence of the double integral, we apply Lemma \ref{lem:cov2}, thereby concluding the proof.
\enp

Now, it  remains to prove Lemmas \ref{lem:exp}, \ref{lem:cov1}, and \ref{lem:cov2}. We prove Lemmas \ref{lem:exp} and \ref{lem:cov1} here, while the proof of Lemma~\ref{lem:cov2}, more technical but based on similar ideas, is deferred to the Appendix \ref{app:A}.
The first step is to provide a usable interpretation of the difference  $\CC_k(u) - \CC_{n, k}(u)$. 
More precisely, this difference can be decomposed into \emph{positive} and \emph{negative} contributions. 
The positive contribution arises  from  $k$-cliques {of $ \PP$} centered at $(0, u)$, having  at least one point lying  outside of $\T_{n}$. The negative contribution comes from $k$-cliques of $\mP_n$, all of whose vertices lie in $\widebar \bbT_n$, but for which there exists a pair of vertices $P_i=(X_i,U_i) \in \widebar \bbT_n$, $i=1,2$, such that $\{X_1, X_2\}$ forms an edge under the toroidal metric but not under the Euclidean metric.
We denote by $\C_{n,k}(u)$ the family of such $k$-cliques contributing either positively or negatively. Note that each $k$-clique in $\C_{n,k}(u)$ contains at least one point $P = (X, U)$ with $|X| \ge n/2$.

Henceforth, to simplify  notation, we set $n_1 := n/8$.
Moreover, when convenient, we will see $ \PP$ and $ \PP_{ n}$ both   as sets and random measures, allowing us to    write $ \PP(A)$ or $  | \PP\cap A | $ for $ A\subset   \overline{   \mathbb{R}}.$ Denote also by $ \ms{Poi}(c)$, a Poisson random  variable with parameter $ c$, or just its law in an abuse of notation. Finally, we remark that $ \PP(A)\sim \ms{Poi}( | A | )$.

%
%
\bep[Proof of Lemma \ref{lem:exp}]
To prove part $(i)$, we consider the highest marked point $Q=(Z, W)$ of the $k$-clique. Apart from a point $ (0,u)$, the remaining $k - 2$ points  of the $k$-clique  lie in  $N^{\da}(Q)$. We recall from Lemma \ref{lem:down}	that the number of down-neighbours of $ Q$ is Poisson distributed with a parameter bounded above by a value independent of $W$ and $n$. Therefore, the number of $ k$-cliques having $(0,u)$ as its lowest marked vertex, is bounded by the number of $ (k-2)$-tuples of points in the down-neighborhood of $Q$. Now, it follows from Lemma \ref{lem:down} that the desired upper bound is given by   
\begin{align}
 \label{eq:neighb-technique}
	{\E \big[\hspace{-.3cm}\sum_{Q\in N^{\ua}(u)\cap \PP}\hspace{-.3cm}\PP(N_{}^{\da}(Q)^{k-2})\big]=\int_{N^\ua(u)}\hspace{-.3cm}\mathbb{E}[\PP(N_{}^{\da}(q))^{k-2}]\d q
	\le\int_{N^{\ua}(u)}\hspace{-.4cm}\E\big[\ms{Poi}(c_-)^{k-2}\big]\d q \in O(u^{-\g})}. 
\end{align}
Next, we continue with part $(ii)$. As stated in the paragraph preceding the proof of Lemma \ref{lem:exp}, it suffices to give an upper bound on the expected number of cliques in $\C_{n,k}(u)$. Then, we distinguish two (non-exclusive) cases. In the first case, the highest marked point $(Z,W)$ is outside of $[-n_1, n_1]\times [0,1]$. In the other case, at least one of the down-neighbors of $(Z, W)$ has at least distance $n_1$ from it.

To ease computations, we assume, without loss of generality, that $\beta = 1$. Note that if $q = (z, w)\in \T_n$ is such that $q \to (0, u)$, with $|z| > n_1$, then  {$ | z |  u^{ \gamma }w^{ 1-\gamma }<1$, and so, $ u = u^{ \gamma  + 1-\gamma }\le u^{ \gamma }w^{ 1-\gamma }<n_{ 1}^{ -1}$}. Hence, in the first case, using the same technique as in  \eqref{eq:neighb-technique} and writing $B(0,n_1)= \{ z \co |z|\le n_1 \}$,  
\begin{align}
	\label{eq:exp1}
	&	\int_0^{n_1^{-1}} u^{-\g} \E\big[\PP(N^\ua(u)\sm B(0, n_1))\big] \d u \le\int_{0}^{n_1^{-1}}u^{-\gamma }\int_{B(0,n_1)^c\times (u,1)}\one\big\{q\to (0,u)\big\}\d q\\
	\notag\le &\int_{0}^{n_1^{-1}}u^{-\gamma }\int_{B(0,n_1)^c\times (u,1)}\one\{ | z | <u^{-\gamma }w^{\gamma -1}\}\d z \d w 
	\le c\int_0^{n_1^{-1}}\int_u^1 u^{-\g}  (u^{-\g}w^{\g - 1} - n_1)_+ \d w \d u\\
	\notag &\le  C\int_0^{n_1^{-1}} u^{-2\g}(1\wedge (u^\g n_1)^{-1/({1 - \g})})^\g \d u.
\end{align}
Hence, there are two cases, namely $u \le n_1^{-1/\g}$ and $u > n_1^{-1/\g}$. In the first case, we obtain an upper bound of  order $cu^{-2\g}\one\{u \le n_1^{-1/\g}\}$. {In this case, integrating over $u$, we obtain an upper bound of order $O(n^{-(1-2\g)/\g})$.} In the second case, {we also deduce from $u\le w$ that $u\le n_1^{-1}$}, and hence, we obtain the upper bound 
$c\one\{ n_1^{-1/\g} \le u \le n_1^{-1}\}n^{-\g/(1-\g)}u^{-2\g - \g^2/(1 - \g)}.$ 
By integrating over $u$, we obtain an upper bound of order $O(n^{-\zeta(\g)}  )$. 

Finally, consider the case where at least one down-neighbor of $(Z, W)$ has distance at least $n_1$. The computation is similar, but we still  need to bound   the integral of  the relevant expression
$u^{-2\g}w^{\g -1}\E\big[\PP(N^\da((z,w))\sm B(z,n_1))\big]$
 over $u \le w \in [0,1]$. 
Note that 
$ \E\big[\PP\big(N^\da((z,w))\sm B(z,n_1)\big)\big] \le w^{\g -1}(w \wedge (w^{ 1- \g}n_1)^{-1/\g})^{1 - \g}.$
If  $w \le n_1^{-1}$, then integrating over $u\le w \le n_1^{-1}$ gives an expression of order $O(n^{-(1 - \g)})$. Second, in the case where $w \ge u \ge n_1^{-1}$,  we obtain an expression of order 
$\int_{n_1^{-1}}^1u^{-1}u^{-(1 - \g)^2/\g} n_1^{-(1 - \g)/\g} \d u \in O(n^{-(1-\g)}).$
Finally, consider the case where $w \ge n_1^{-1} \ge u$. Then, we obtain 
$n_1^{-\g}\int_0^{n_1^{-1}} u^{-2\g} \d u$,  which is again of order  $O(n^{-(1- \g)})$, thereby concluding the proof.
\enp

%
%
%
%

%
%
\bep[Proof of Lemma \ref{lem:cov1}]
We reduce the argument  to finding the desired bounds as  in the proof of Lemma \ref{lem:exp}. 
Specifically, we again  decompose $\CC_k(u) - \CC_{n,k}(u)$ into its positive and negative contributions. In addition, it is    necessary to count the number of pairs consisting of a $k$-clique in $\C_{n,k}(u)$ and     an additional $\ell$-clique centered at $(0,u)$,   such that the two cliques share at least one common vertex.  
If the two cliques in a pair   share the same highest marked vertex  $(Z,W)$, then all other nodes must lie   in the down-neighborhood of $(Z,W)$,   and the assertion follows as in the proof of Lemma~\ref{lem:exp}. 
Otherwise, the two cliques     have distinct highest marked vertices, denoted respectively as $(Z,W)$ and $(Z',W')$. In this case, the down-neighbors of $(Z,W)$ and $(Z',W')$ are both Poisson distributed with constant parameters. The contribution from the additional $\ell$-clique can still  be bounded above by $Cu^{-\g}$, and the rest of the argument is the same as the proof of Lemma \ref{lem:exp}. 
\enp

\section{Proof of Proposition \ref{pr:gisu} -- CLT for clique counts}
\label{sec:cliq}

In this section, we prove Proposition \ref{pr:gisu}. We begin with the analysis of the $\Ga_3'$-term. This term needs a simpler analysis than $\Ga_1'$ and $\Ga_2'$, as it involves only the moments of first-order differences. We then proceed to handle $\Ga_1'$ and $\Ga_2'$, which are substantially more delicate,  due to the need to control moments of second-order differences. To ease notation, we henceforth assume $\beta = 1$.
%
%
\bepr[Moment bound for $D_u\CC_{n, k}$]
\label{pr:dpc}
Let $k_0 \ge 1$ and  $a > 1$ be such that $a\g < 1$. Then, 
 $$\limsup_{u \to 0} \sup_{n \ge 1}\max_{k \le k_0}\f{\log\E[(D_u\CC_{n, k})^a]}{\log(1/u)} \le a\g.$$
\enpr

Before proving Proposition \ref{pr:dpc}, we explain how to deduce the rate of  convergence of  $\Ga_3'$.

%
%
\bep[Proof of Propositions \ref{pr:gisu}; $\Ga_3'$-term at \eqref{Ga3}.]  
Choose $\eta \in (1,2)$ with $(\eta + 1)\g < 1$. 
By Proposition \ref{pr:dpc}, we see that $\sup_{n \ge 1}\E[(D_u\CC_{n, k})^{\eta+ 1}] \le c u^{-(\eta + 1)\g + o(1)}$. Hence, we have, as asserted, that $ \sup_{n \ge 1}\int_0^1\E[(D_{u}\CC_{n, k})^{\eta + 1}]\d u < \ff. 
$

\enp

To derive the   rates of convergence for  $\Ga_1'$ and $\Ga_2'$, we need bounds for the moments of the second-order differences. To state them precisely, we introduce the function 
\begin{align}
	\label{eq:sur}
	s_\wedge(u, r) := 1\wedge (u^\g r)^{-1/(1 - \g)} \one\{r \le 2u^{-1}\}, \ \ \ u\in (0,1), \ \ r >0. 
\end{align}

%
%
\bepr[Moment bound for $D_{u, q}\CC_{n, k}$]
\label{pr:dpqc}
Let $a > 1$ be such that $a\g < 1$ and $k_0 \ge 1$. Then, there exists a constant $C = C(a, \g) >0$ such that for $0 \le u \le v \le1$ and $q=(y,v)\in \widebar \bbT_n$, 
$$
\max_{k\le k_0}\E[(D_{u,q}\CC_{n, k})^a] \le C\big(\one \{q \in N^{\ua}(u)\} + v^{-a\g} s_\wedge(u, | y|)^{a \g}\big). 
$$
\enpr

We also need a specific integral computation as follows. 

%
%
\bel
\label{lem:intcomp}
Let $ a\in \{1,2\}$. Then,
$\int_{\R}s_\wedge(u, |y|)^{a\g} \d y \in O(u^{-(\g \vee (1 - a\g))}\log(1/u))$.
\enl
\bep
Since $|[x- u^{-\g}, x+ u^{-\g}]| =2u^{-\g}$, we may assume that $| y| \ge u^{-\g}$.  Also, $s_\wedge(u, |y|)^{a\g} \le  u^{-a\g^2/(1 - \g)} |y|^{-a\g/(1 -\g)} $.
We first assume $a\g/(1 - \g) > 1$ so that 
$\int_1^{\ff} r^{-a\g/(1- \g)} \d r < \ff.$
Hence, using the change of variables $r:= u^\g | y| $, we have that
$\int_{\R}s_\wedge(u,|y|)^{a\g}\d y\in O(u^{-\g}).$
Second, we consider the case $a\g/(1 - \g) < 1$. Then, by \eqref{eq:sur}, we  have $| y| \le2 u^{-1}$. Moreover,
$u^{-a\g^2/(1-\g)}\int_0^{2u^{- 1}} r^{-a\g/(1- \g)} \d r  \in O(u^{-(1 - a\g)}),$
as asserted. The case $a\g=(1 - \g)$ is similar except for the appearance of a new $\log (1/u)$ term. 
\enp

By Propositions \ref{pr:dpc} and \ref{pr:dpqc}, we can obtain   the asymptotic bounds for  $\Ga_1'$ and $\Ga_2'$.  For the rest of this section, we choose $\eta \in (1,2)$ such that $\eta (2\g \vee (1 - \g)) < 1$. We henceforth put $\tau(\g) = \g \vee ( 1- 2\g)$.

%
%
\bep[Proof of Proposition \ref{pr:gisu}; Term at \eqref{Ga2}]
We need to control the expression  
\begin{align}
	\label{eq:gisui}
	\int_{0}^1 \left(\int_{\T_n} \max_{k \le k_0} \E\left[(\MalD_{u,q} \CC_{n, k})^{2\eta} \right]^{1/\eta} \d q \right)^\eta \d u, 
\end{align}   
where $q=(y,v)$. We first assume that $u \le v$.
Then, by Proposition \ref{pr:dpqc} with $a=2\eta$, 
\begin{align}
	\label{eq:p12}
	\max_{k\le k_0}\E\left[( \DD_{ u, q}\CC_{n, k})^{2\eta}\right]^{1/\eta} \le C\big(\one \{q \in N^{\ua}(u)\} + v^{-2\g} s_\wedge(u, |y|)^{2\g}\big).
\end{align}
For the inner integral, note that $|N^{\ua}(u)|\in O( u^{-\g})$. Also, we can apply Lemma \ref{lem:intcomp}, which gives that 
$\int_{\R} s_\wedge(u,|y|)^{2\gamma} \d y \in O\big(u^{-\tau(\gamma)}\log(1/u)\big)$. Hence, we bound \eqref{eq:gisui} as 
$\int_0^1u^{-\tau(\gamma)\eta}(\log u^{-1})^\eta \d u <\ff$. 

Next, we assume $v < u$. Then, \eqref{eq:p12} holds with the roles of $u$ and $v$ reversed. Note that $|N^{\da}(u)|\le C$. Moreover,
 we apply Lemma \ref{lem:intcomp} to compute the integral over $y$, resulting in 
$$\int_0^1\int_{\R}s_\wedge(v, |y|)^{2\g} \d y\d v \le c\int_0^1 v^{-\t(\g)}\log(1/v) \d v< \ff.$$
Since 
$\int_0^1u^{-2\g\eta} \d u<\ff$, we can conclude  the proof.
\enp

The proof of the $\Ga'_1$-term is similar, but for convenience of the reader, we still provide some details.
%
%
\bep[Proof of Proposition  \ref{pr:gisu}; Term at \eqref{Ga1}]
First, note that $\E\left[(D_q\CC_{n, k})^{2\eta}\right]^{1/(2\eta)} \in O(v^{-\g})$  by Propositions \ref{pr:dpc}. Now, we first consider the case where $u \le v$. We note that integrating $u^{-\g}v^{-1}$ for $v \ge u$ gives an expression of order $O\big(u^{-\g}\log(1/u)\big)$. This quantity is integrable, even when raised to the power $\eta$.
Hence, by Proposition \ref{pr:dpqc} with $a=2\eta$ it remains to deal with the quantity
\begin{align}
	\label{eq:g1}
	 Cv^{-2\g} s_\wedge(u, |y|)^\g.
\end{align}
Now, we integrate over $y$, applying  Lemma \ref{lem:intcomp}. This yields a term of order $O\big(u^{-(1-\g)}\log(1/u)v^{-2\g}\big)$. Hence, integrating over $u$ and $v$ shows the asserted finiteness, even after raising to the power $\eta$.

Next, we deal with the case $v<u$. Here, we need to first consider the integral of $v^{-2\g}u^{\g - 1}$ over $v \le u$, which is of order $O(u^{-\g})$. Hence, this quantity is integrable even after raising to the power $\eta$. Now, analogous to \eqref{eq:g1}, we need to deal with  $Cu^{-\g}v^{-\g} s_\wedge(v, |y|)^\g$. If $|y| \le v^{-\g}$, then we obtain an expression of order $Cu^{-\g}v^{-2\g}$. Hence, we can from now on assume that $|y|>v^{-\g}$. 
Again, we can integrate out $y$ using Lemma \ref{lem:intcomp} and are reduced to the quantity $Cu^{-\g}v^{-1}$. Now, integrating $v$ in the domain $n^{-1/\g} \le v \le u$ gives an expression of order $C u^{-\g} \log(n^{1/\g}u)$. Raising to the power $\eta$ and integrating, we obtain a quantity of order $(\log n)^\eta$.
\enp

%
%
\subsection{Proof of Propositions \ref{pr:dpc} and \ref{pr:dpqc}}
\label{ss:dpqc}

It remains to prove Propositions \ref{pr:dpc} and \ref{pr:dpqc}. For both cases, a key ingredient is the following moment bound on the maximal degree, defined by 
$$D_{\max}(u) := \max_{P \in \mP\cap N^\ua(u)}\PP(N^\da(P)), \ \ \ u \in (0,1).$$

%
%
\bel[Moment bound for $D_{\max}(u)$]
\label{lem:dm}
Let $a \ge 1$. Then,
$ {\log\E[D_{\max}(u)^a]}\in o({\log(1/u)}).$
\enl
\bep
By the Cauchy-Schwarz inequality, for every $ \varepsilon >0,$
\begin{align*}
	\E[D_{\max}(u)^a] &\le u^{-\e} + \E\big[D_{\max}(u)^a \one\{D_{\max}(u) > u^{-\e/a}\}\big].
\end{align*}
By bounding the max by the sum, together with the Mecke formula, the expectation on the right-hand side is bounded above, up to a multiplicative constant, by  
\begin{align*}
&\E\Big[\sum_{Q\in \PP}\one{\{Q \to (0,u)\}} \PP( N^{\da}(Q) )^a \one{\{ \PP( N^{\da}(Q) ) >u^{-\varepsilon /a}\}}\Big]
	\\
	 &= \int_{  \RR}\one \{q\to (0,u)\} \E\big[ \PP( N^\da(q) ) ^a\one \{ \PP( N^\da(q) )  > u^{-\e/a}\big\}\big] \d q\\
	 &\le \E \big[N^a\one{\{N>u^{-\varepsilon /a}\}}\big]\E \big[\PP(N^{\ua}(u))\big],
\end{align*}
where $ N\sim {\ms{ Poi}}(c_-)$ (see Lemma \ref{lem:down}). Here, the first expectation of the rightmost term is bounded by $c_1\exp(-c_2u^{-\e/a})$ for some $c_1, c_2> 0$. 
By  Lemma \ref{lem:down}, we also have  $\E[\PP(N^{\ua}(u))] \in O(u^{-\g})$, which concludes the proof.
\enp

It now remains to complete the proofs of  Propositions \ref{pr:dpc} and \ref{pr:dpqc}.

%
%
\bep[Proof of Proposition \ref{pr:dpc}] 
First, for each $1\le \ell \le k$, we write $\CC_{n, k}^{(\ell)}(u)$ for the number of $k$-cliques in which  $ (0,u)\in \widebar \bbT_n$ is  its $\ell$th lowest marked vertex.  
Then, we claim that for every $1\le \ell \le k \le k_0$, 
\begin{align}
	\label{eq:cnmk}
		\limsup_{u \to 0}\sup_{n \ge 1}\f{\log\E[\CC_{n, k}^{(\ell)}(u)^a]}{\log (1/u)} \le a\g.
\end{align}
		In particular, \eqref{eq:cnmk}  immediately implies  Proposition \ref{pr:dpc}. 

We focus here on the   case $1 \le \ell \le k - 1$. For any $k$-clique contributing to $\CC_{n, k}^{(\ell)}(u)$, there exists   a point $P \in \PP_n$ with the highest mark within this $k$-clique,  such that   $P \to (0, u)$ and the remaining $k - 2$ points of that clique lie in $N^\da(P)$.   Thus, it follows that 
\begin{align}
	\label{eq:cnmka}
	\CC_{n, k}^{(\ell)}(u) \le 
	\PP(N^\ua(u))\max_{P\in \PP \cap N^{\ua}(u) }\PP(N^\da(P))^{k-2}.
\end{align}
Hence, for any $b_1, b_2 >0$ with $1/b_1 + 1/b_2 = 1$,  the H\"older inequality gives that 
$$ 
\E[\CC_{n, k}^{(\ell)}(u)^a] \le \E[\PP(N^\ua(u))^{ab_1}]^{1/b_1}\E[\max_{P\in \PP \cap N^{\ua}(u)}\PP(N^\da(P))^{ab_2(k-2)}]^{1/b_2}.
$$
By Lemma \ref{lem:down}, the first expectation is of order $O(u^{-a\g})$, and by Lemma \ref{lem:dm}, the second term is of logarithmic order, thereby concluding the proof of \eqref{eq:cnmk}.  
\enp
\smallskip

%
%
\bep[Proof of Proposition \ref{pr:dpqc}] Throughout the proof, we write $q=(y,v)$. 
Our envisioned bound will be derived from 
$(D_{u, q}\CC_{n, k})^a \le G_1 + G_2,$
where the split into two terms depends on whether the highest marked point $(Z, W)$ of the considered $k$-clique equals $q$.  We need to discuss the following  two cases: 
\been
\im $(Z, W) = q$.  In this case, it follows that 
$G_1 \le \one \{q \in N^{\ua}(u)\}\PP(N^\da(q))^{a(k-2)}$, 
so that 
$$\E[G_1] \le \one \{q \in N^{\ua}(u)\} \E[\PP(N^\da(q))^{a(k-2)}]. $$
By Lemma \ref{lem:down}, the expectation on the right-hand side is of finite order. 
\im $(Z, W) \ne q$.  We let 
$N^\ua(u, q) := \{r\in \widebar \bbT_n \co r \to (0, u), \, r \to q\}$
denote the collection  of points connecting toward two nodes $(0, u)$,  $q\in\widebar  \bbT_n$. Then, as in \eqref{eq:cnmka}, the number of $k$-cliques containing $(0, u)$ and $q$, such that the  highest marked point lies above $(0,u)$ and $q$, is at most 
$\PP\big(N^\ua(u, q)\big)^aD_{\max}(q)^{a(k-3)}.$
Now, by applying the H\"older inequality for $b_1, b_2 > 0$ with $1/b_1 + 1/b_2 = 1$, 
\begin{equation}  \label{e:EG_2}
\E[G_2] \le \E[\PP(N^\ua(u, q))^{ab_1}]^{1/b_1} \E[D_{\max}(q)^{ab_2(k-3)}]^{1/b_2} 
\end{equation}
Observe that $\PP(N^\ua(u, q))$ is a Poisson random variable  with parameter bounded by $Cv^{-\g}s_\wedge(u, | y|)^\g$. Hence, 
$\E[\PP(N^\ua(u, q))^{ab_1}]^{1/b_1} \le Cv^{-a\g}s_\wedge(u, | y|)^{a\g}.$
Furthermore, by Lemma \ref{lem:dm}, the second expectation in \eqref{e:EG_2} is of order $o(\log(1/v))$.  Combining these two observations gives the asserted bound.
\enen
\enp

\section{Proof of Theorem \ref{t:CLT.sub-tree} -- CLT for subtree counts}
\label{sec:tree}
We begin with the asymptotic bounds for the moment of $\inDt(x,u)$. 
The proof of the following proposition  is deferred to the Appendix \ref{app:B}. 
First, we define the expectation of $\inDt(x,u)$  by 
$$
\mu_{n, \ms T}(u) := \E \big[ \inDt(x,u) \big], \ \ \ (x,u) \in \widebar \bbT_n. 
$$
Due to the homogeneity of $\mP_n$, we have dropped ``$x$" from the left-hand side.

\begin{proposition}  \label{p:exp.var.sub-tree}
Let $0< \gamma < 1/(2\ell)$, $\eta >0$ and $\ms{T}$ be a directed tree. Then, there exist $N\in \N$ and $C\in (0,\infty)$, such that  for all $n\ge N$ and $u\in (0,1)$, the following properties hold:
$(i)$
$
\mu_{n,\ms{T}}(u)\le  C u^{-\ell \gamma-\eta},
$
$(ii)$
$\ms{Var} \big( \inDt(u) \big) \le C u^{-2\ell\gamma-\eta}, 
$
and 
$(iii)$ 
$
\E \big[ \inDt(u)^3 \big] \le C u^{-3\ell\gamma -\eta}.
$
Furthermore, define 
$
\inD(u) := \lim_{n\to\infty} \inDt(u), \ \ \text{ and } \ \ \mu_{\ms{T}}(u):=\lim_{n\to\infty} \mu_{n,\ms{T}}(u); 
$
then, $\mu_{\ms{T}}(u)$, $\ms{Var} \big( \inD(u) \big)$, and $\E \big[ \inD(u)^3 \big]$ satisfy the same  bounds as in $(i)$ -- $(iii)$ above. 
\end{proposition}

\begin{proof}[Proof of Lemma \ref{l:cov.var.bdd.tree}]
\underline{\textit{Proof of $(i)$}}: For ease of description, we will establish an upper bound for $\ms{Cov}(T_1^{(n)}, \, T_{k+2}^{(n)})$. 
First, the Mecke formula yields that 
\begin{align}
\ms{Cov}&(T_1^{(n)},T_{k+2}^{(n)}) = \int_{[0,1]^2} \int_{[k+1,k+2] \times [0,1]}\ms{Cov} \big( \inDt(x,u), \inDt(y,v) \big)\dif y \dif v \dif x \dif u + R_k^{(n)} \label{e:cov.T1.Tk} \\
&= \int_{[0,1]^2} \int_{[k+1,k+2] \times [0,1]} \E \bigg[ \sum_{(P_1, \dots, P_{m-1})\in (\Pn)_{\neq}^{m-1}}\sum_{\substack{(Q_1,\dots,Q_{m-1})\in (\Pn)_{\neq}^{m-1}, \\ |(P_1,\dots,P_{m-1})\cap (Q_1,\dots,Q_{m-1})|\ge 1}} \hspace{-10pt}h\big( (x,u),P_1,\dots,P_{m-1} \big)\,\notag \\
&\qquad \qquad \qquad \qquad \qquad \qquad \qquad \qquad  \times   h\big( (y,v),Q_1,\dots,Q_{m-1} \big)\bigg]  \dif y \dif v \dif x \dif u + R_k^{(n)},  \notag 
\end{align}
where $h$ is the  indicator function for the event that there exists a graph homomorphism from $\ms{T}$ to an induced graph defined on  $\big\{ (0,u), P_1,\dots, P_{m-1} \big\}$, with the root $\ms{r}$  mapped to $(0,u)$. See also \eqref{e:expression.Din} in the Appendix. Moreover, $R_k^{(n)}$ is a negligible remainder term; see \cite[Lemma 10]{hirsch:juhasz:2023} for its detailed calculations.

As argued in \cite{ht24} (see the proof of Proposition 2.2 therein), it suffices to consider the case $\big|(P_1,\dots,P_{m-1})\cap (Q_1,\dots,Q_{m-1})\big|= 1$. Let $\ms{T}_1$ and $\ms{T}_2$ be two identical trees isomorphic to $\ms{T}$, and select a node from each of $\ms{T}_1$ and $\ms{T}_2$, and identify these nodes  together. For the calculation of \eqref{e:cov.T1.Tk}, it is essential to consider the configuration of nodes and edges, defined by 
\begin{align} 
\begin{split}  \label{e:tree.light.tail}
&(z_0,w_0) \to (z_1,w_1) \to \cdots \to (z_d,w_d) \to (x,u), \\
&(z_0,w_0) \to (z_1',w_1') \to \cdots \to (z_e',w_e') \to (y,v), 
\end{split}
\end{align}
where $d, e \ge 0$, and $(z_0,w_0)$ corresponds to the common node between $(P_1,\dots, P_{m-1})$ and $(Q_1, \dots, Q_{m-1})$. 
If $d=0$ (resp.~$e=0$), then $(z_0,w_0)$ directly connects to $(x,u)$ (resp.~$(y,v)$) without any intermediate nodes. To clarify, we assume that the nodes in the first line of \eqref{e:tree.light.tail} are taken from $\ms{T}_1$, while all the nodes in the second line belong to $\ms{T}_2$. 

As in the proof of Proposition \ref{p:exp.var.sub-tree} in the Appendix, we introduce a non-negative integer $\ell_0$ as the number of leaves that have no further children and are connected to $(z_0,w_0)$ by the paths consisting of nodes whose marks are all higher than $w_0$. Additionally, for every $i=1,\dots,d$ (resp.~$j=1,\dots,e$), there are $\ell_i$ (resp.~$\ell_j'$) leaves that have no further children and are connected to $(z_i,w_i)$ (resp.~$(z_j',w_j')$) by the paths  not including $(z_{i-1}, w_{i-1})$ (resp.~$(z_{j-1}', w_{j-1}')$), but consisting of nodes whose marks are all higher than $w_i$ (resp.~$w_j'$). Moreover, let $p$ (resp.~$p'$) denote the number of leaves that have no further children and are connected to $(x,u)$ (resp.~$(y,v)$) by the paths  not including $(z_d,w_d)$ (resp.~$(z_e',w_e')$), but consisting of nodes whose marks are higher than $u$ (resp.~$v$). 

In this setting, if the common node $(z_0,w_0)$ is not taken from any leaves of either $\ms{T}_1$ or $\ms{T}_2$, then 
\begin{equation}  \label{e:ell.and.other.indices1}
2\ell = p + p' + \ell_0 + \sum_{i=1}^d \ell_i + \sum_{j=1}^e \ell_j'.
\end{equation}
If $(z_0,w_0)$ corresponds to a leaf of either $\ms{T}_1$ or $\ms{T}_2$ but not for the other, then we have 
\begin{equation}  \label{e:ell.and.other.indices2}
2\ell = p + p' + \ell_0 + \sum_{i=1}^d \ell_i + \sum_{j=1}^e \ell_j' + 1.
\end{equation}
Finally, if $(z_0,w_0)$ is taken from the leaves of both $\ms{T}_1$ and $\ms{T}_2$, then $\ell_0=0$, so that 
\begin{equation}  \label{e:ell.and.other.indices3}
\ell = p + \sum_{i=1}^d \ell_i + 1 = p' + \sum_{j=1}^e \ell_j' + 1.
\end{equation}
Notice that if $\ell_0 \ge 2$, either \eqref{e:ell.and.other.indices1} or \eqref{e:ell.and.other.indices2} applies, whereas \eqref{e:ell.and.other.indices2}  always holds when $\ell_0 = 1$. Moreover, \eqref{e:ell.and.other.indices3} is valid only when $\ell_0 = 0$. 

In the sequel, we focus only on the case $\ell_0 \ge 1$, as $\ell_0 = 0$ is an easier case that can be treated analogously. We assume further that either $d$ or $e$ is positive, because the case $d = e = 0$ is significantly simpler to study. Without loss of generality, we take $d \ge 1$.  In the setup and notation above, we note 
$$
\max_{0\le i \le d-1} \ms{dist}_n(z_i,   z_{i+1}) \vee \ms{dist}_n(z_d, x)    \vee \ms{dist}_n(z_0, z_1') \vee \max_{1\le j \le e-1} \ms{dist}_n(z_j', z_{j+1}') \vee \ms{dist}_n(z_e', y) \ge \frac{k}{d+e+2}, 
$$
where $\ms{dist}_n(x,y) := \min_{z\in \Z} |x-y+nz|$ denotes the toroidal metric on $\bbT_n$. 
By the symmetry between $(z_i)$ and $(z_j')$, along with the assumption $d \ge 1$, it is sufficient to consider the following two cases:
\begin{align}
\begin{split}  \label{e:case.i.and.case.ii}
&\text{Case} \, (i): \ \ms{dist}_n(z_q, z_{q+1}) \ge\frac{k}{d+e+2} \text{ for some } q=0,\dots,d-1, \\
&\text{Case} \, (ii): \ \ms{dist}_n(z_d, x) \ge \frac{k}{d+e+2}. 
\end{split}
\end{align}
The required arguments in both cases are similar, so we discuss Case $(i)$ only. In Case $(i)$, we have 
$$
\frac{k}{d+e+2} \le \ms{dist}_n(z_q, z_{q+1}) \le \beta w_{q+1}^{-\gamma} w_q^{\gamma-1} \le \beta u^{-\gamma} w_q^{\gamma-1} \le \beta u^{-1}, 
$$
and thus, 
\begin{equation}  \label{e:bdd.w0.u}
w_q\le  \Big( \frac{(d+e+2)\beta}{ku^\gamma} \Big)^{1/(1-\gamma)}\, \one \Big\{  u \le \frac{(d+e+2)\beta}{k} \Big\} =: s(u,k). 
\end{equation}
	Define also $s_\wedge (u,k) := s(u,k)\wedge 1$. 
	For simplicity, we may set, without loss of generality, $(d+e+2)\beta \equiv 1$ and then recover the definition from Section \ref{sec:cliq}.

In this setting, we need to prove that the following  integral is of order  $O(k^{-(2 - 2\ell \g)})$ uniformly over $n$: 
\begin{align}
\begin{split}  \label{e:initial.int}
&\int_{u=0}^{1} \int_{v=0}^1 \int_{w_e'=v}^1 \int_{w_{e-1}'=w_e'}^1 \cdots \int_{w_1'=w_2'}^1 \int_{w_0=w_1'}^1 \\
&\quad \times \int_{w_d=u}^1 \int_{w_{d-1}=w_d}^1 \cdots \int_{w_1=w_2}^1 w_0^{-\ell_0\gamma} \Big( \prod_{i=1}^d w_i^{-\ell_i\gamma}  \Big) \Big( \prod_{j=1}^e (w_j')^{-\ell_j'\gamma} \Big)u^{-p\gamma} v^{-p'\gamma} \one \{ w_q\le s_\wedge(u,k) \} \\
&\quad \times \int_{x=0}^1 \int_{y=k+1}^{k+2} \prod_{j=1}^e \int_{z_j'\in [0,n]} \int_{z_0\in [0,n]} \prod_{i=1}^d \int_{z_i\in [0,n]} \one \big\{ (z_0,w_0) \to (z_1,w_1) \to \cdots \to (z_d,w_d) \to (x,u) \big\} \\
&\qquad \qquad \qquad \qquad \qquad \qquad\qquad \qquad  \qquad \times  \one \big\{ (z_0,w_0) \to (z_1',w_1') \to \cdots  \to (z_e',w_e') \to (y,v) \big\}. 
\end{split}
\end{align}
Here, the integrand in the second line is obtained by repeatedly integrating out both spatial and time coordinates of all the vertices, except for those in \eqref{e:tree.light.tail}. As detailed after \eqref{e:A.u} in the Appendix, additional logarithmic factors may need to be multiplied into the above integrands. However, in such cases, it suffices to bound these logarithmic factors by terms such as $w_0^{-\eta \gamma}$, etc. For simplicity of presentation, we will ignore the appearance of logarithmic factors throughout the proof.

We start with the case $q=0$. Dropping the condition $(z_0,w_0) \to (z_1,w_1)$ from \eqref{e:initial.int} (with $q=0$), we calculate integrals over part of the variables in \eqref{e:initial.int}: 
\begin{align}
A_{u,k}: &= \int_{v=0}^{s_\wedge(u,k)} \int_{w_e'=v}^{s_\wedge(u,k)} \int_{w_{e-1}'=w_e'}^{s_\wedge(u,k)} \cdots \int_{w_1'=w_2'}^{s_\wedge(u,k)} \int_{w_0=w_1'}^{s_\wedge(u,k)} w_0^{-\ell_0\gamma} \prod_{j=1}^e (w_j')^{-\ell_j'\gamma} v^{-p'\gamma} \label{e:A.uk}\\
&\quad \times \int_{y=k+1}^{k+2} \prod_{j=1}^e \int_{z_j'\in [0,n]} \int_{z_0\in [0,n]} \one \big\{ (z_0,w_0) \to (z_1',w_1') \to \cdots  \to (z_e',w_e') \to (y,v) \big\}. \notag
\end{align}
Recall that as in \eqref{e:A.u} in the Appendix, we have suppressed the dependence on $n$ from $A_{u,k}$. 
By integrating out all spatial coordinates in \eqref{e:A.uk}, as well as using  the bound of the type \eqref{e:torus.bdd} in the Appendix, 
\begin{align*}
A_{u,k} &\le  \int_{v=0}^{s_\wedge(u,k)} v^{-(p'+1)\gamma}  \int_{w_e'=v}^{s_\wedge(u,k)} (w_e')^{-1-\ell_e'\gamma} \int_{w_{e-1}'=w_e'}^{s_\wedge(u,k)} (w_{e-1}')^{-1-\ell_{e-1}'\gamma}\times \cdots  \\
&\qquad \qquad \qquad\qquad \qquad \times \int_{w_1'=w_2'}^{s_\wedge(u,k)}  (w_1')^{-1-\ell_1'\gamma} \int_{w_0=w_1'}^{s_\wedge(u,k)} w_0^{-1-(\ell_0-1)\gamma}. 
\end{align*} 
Since we have assumed $\ell_0\ge1$, it holds that 
$
 \int_{w_0=w_1'}^{s_\wedge(u,k)} w_0^{-1-(\ell_0-1)\gamma}\le C(w_1')^{-(\ell_0-1)\gamma}. 
$
By carrying out all the remaining integrals, 
$$
A_{u,k} \le C\int_{v=0}^{s_\wedge(u,k)} v^{-(p'+\ell_0+\sum_{j=1}^e \ell_j')\gamma} =  C s_\wedge(u,k)^{1-(p'+\ell_0+\sum_{j=1}^e \ell_j')\gamma}, 
$$
where the above equality is assured by the constraint $\gamma<1/(2\ell)$. 

We continue by calculating integrals over the remaining coordinates, i.e., $(z_1,w_1), \dots, (z_d,w_d)$, and $(x,u)$ in \eqref{e:initial.int}, to estimate the expression 
\begin{align*}
B_k &:= \int_{u=0}^{1}\int_{w_d=u}^{s_\wedge(u,k)} \int_{w_{d-1}=w_d}^{s_\wedge(u,k)} \cdots \int_{w_1=w_2}^{s_\wedge(u,k)} u^{-p\gamma} s_\wedge(u,k)^{1-(p'+\ell_0+\sum_{j=1}^e \ell_j')\gamma} \prod_{i=1}^d w_i^{-\ell_i\gamma} \\
&\quad \times \int_{x=0}^1 \prod_{i=1}^d \int_{z_i\in [0,n]} \one \big\{ (z_1,w_1) \to \cdots \to (z_d,w_d) \to (x,u) \big\}  \\
&\le C  \int_{u=0}^{1} u^{-(1+p)\gamma} s_\wedge(u,k)^{1-(p'+\ell_0+\sum_{j=1}^e \ell_j')\gamma} \int_{w_d=u}^{s_\wedge(u,k)} w_d^{-1-\ell_d\gamma} \int_{w_{d-1}=w_d}^{s_\wedge(u,k)} w_{d-1}^{-1-\ell_{d-1}\gamma} \times \cdots \\
&\qquad \qquad\qquad \qquad\qquad \qquad  \times \int_{w_2=w_3}^{s_\wedge(u,k)} w_2^{-1-\ell_2\gamma}\int_{w_1=w_2}^{s_\wedge(u,k)} w_1^{-1-(\ell_1-1)\gamma}. 
\end{align*}
Observe that 
\begin{equation}  \label{e:inner.bound1}
\int_{w_1=w_2}^{s_\wedge(u,k)} w_1^{-1-(\ell_1-1)\gamma} \le \begin{cases}
Cw_2^{-(\ell_1-1)\gamma} & \text{if } \ell_1\ge1, \\ 
Cs_\wedge(u,k)^\gamma & \text{if } \ell_1=0. 
\end{cases}
\end{equation}
If $\ell_1\ge1$, it follows from  the first bound in \eqref{e:inner.bound1}, as well as   \eqref{e:bdd.w0.u}, that 
\begin{align*}
B_k &\le C\int_{0}^{1} u^{-(p+\sum_{i=1}^d \ell_i)\gamma} s(u,k)^{1-(p'+\ell_0+\sum_{j=1}^e \ell_j')\gamma} \dif u \\
&= C\int_0^{1/k} u^{-(p+\sum_{i=1}^d \ell_i)\gamma} \Big( \frac{1}{ku^\gamma} \Big)^{\frac{1-(p'+\ell_0+\sum_{j=1}^e \ell_j')\gamma}{1-\gamma}} \dif u \\
&=Ck^{-(2-(p+p'+\ell_0+\sum_{i=1}^d\ell_i +\sum_{j=1}^e \ell_j')\gamma)}  \le   Ck^{-(2-2\ell \gamma)}. 
\end{align*}
The last step above holds as an equality if \eqref{e:ell.and.other.indices1} is satisfied. On the other hand, under the condition \eqref{e:ell.and.other.indices2}, the last step is a strict inequality.  
If $\ell_1 = 0$, using the second bound in \eqref{e:inner.bound1}, we observe that  
\begin{align*}
B_k &\le C \int_{0}^{1} u^{-(1+p+\sum_{i=1}^d \ell_i)\gamma} s(u,k)^{1-(p'-1+\ell_0+\sum_{j=1}^e \ell_j')\gamma} \dif u\\
&=  C\int_0^{1/k} u^{-(1+p+\sum_{i=1}^d \ell_i)\gamma} \Big( \frac{1}{ku^\gamma} \Big)^{\frac{1-(p'-1+\ell_0+\sum_{j=1}^e \ell_j')\gamma}{1-\gamma}} \dif u \\
&=Ck^{-(2-(p+p'+\ell_0+\sum_{i=1}^d\ell_i +\sum_{j=1}^e \ell_j')\gamma)}  \le   Ck^{-(2-2\ell \gamma)}. 
\end{align*}

Now, let us return to Case $(i)$ in \eqref{e:case.i.and.case.ii} and examine the remaining case  $q \in \{1, \dots, d-1\}$. Once again, our goal is to show that  the integral in \eqref{e:initial.int} is bounded by the term of order  $O(k^{-(2-2\ell\g)})$ uniformly over $n$.  We will consider the following distinct cases, based on the size of the time coordinates. Once these cases have been treated, the proof of the assertion will be complete: 
\begin{enumerate}
	\item[] (I) $v > 1/k$,
	\item[] (II) $w_e' > 1/k$ and $v \le 1/k$,
	\item[] (III) $w_{s-1}' > 1/k$ and $w_s' \le 1/k$ for some $s = 2, \dots, e$,
	\item[] (IV) $w_0 > 1/k$ and $w_1' \le 1/k$, 
	\item[] (V) $w_0 \le 1/k$.
\end{enumerate}
We first discuss Cases (I)--(IV), which always indicates that $w_0>1/k$. Removing the condition $(z_q,w_q)\to (z_{q+1}, w_{q+1})$ from \eqref{e:initial.int}, we first need to bound the  integral: 
\begin{align*}
C_{u,k} &:= \int_{v=0}^1 \int_{w_e'=v}^1 \int_{w_{e-1}'=w_e'}^1 \cdots \int_{w_1'=w_2'}^1 \int_{w_0=w_1'\vee \frac{1}{k}}^1 \int_{w_q=0}^{s_\wedge(u,k)} \int_{w_{q-1} = w_q}^1 \cdots \int_{w_1=w_2}^1 \\
&\qquad \qquad w_0^{-\ell_0\gamma} \Big( \prod_{j=1}^e (w_j')^{-\ell_j'\gamma}\Big) \Big( \prod_{i=1}^q w_i^{-\ell_i\gamma} \Big) v^{-p'\gamma} \\
&\quad \times \int_{y=k+1}^{k+2}  \prod_{j=1}^e \int_{z_j'\in [0,n]} \int_{z_0\in [0,n]} \one \big\{ (z_0,w_0) \to (z_1',w_1') \to \cdots  \to (z_e',w_e') \to (y,v) \big\} \\
&\quad \times \prod_{i=1}^q \int_{z_i\in [0,n]} \one \big\{ (z_0,w_0) \to (z_1,w_1) \to \cdots \to (z_q,w_q)  \big\}. 
\end{align*}
By integrating the above expression over all spatial coordinates, 
\begin{align*}
&\int_{y=k+1}^{k+2}  \prod_{j=1}^e \int_{z_j'\in [0,n]} \int_{z_0\in [0,n]} \one \big\{ (z_0,w_0) \to (z_1',w_1') \to \cdots  \to (z_e',w_e') \to (y,v) \big\} \\
&\quad \times \prod_{i=1}^q \int_{z_i\in [0,n]} \one \big\{ (z_0,w_0) \to (z_1,w_1) \to \cdots \to (z_q,w_q)  \big\} \\
&\le C w_0^{2\gamma-2} \Big( \prod_{i=1}^{q-1} w_i^{-1} \Big) w_q^{-\gamma} \prod_{j=1}^e (w_j')^{-1} v^{-\gamma}, 
\end{align*}
and hence, 
\begin{align*}
C_{u,k} &\le  C \int_{v=0}^1 v^{-(1+p')\gamma} \int_{w_e'=v}^1 (w_e')^{-1-\ell_e'\gamma} \int_{w_{e-1}'=w_e'}^1 (w_{e-1}')^{-1-\ell_{e-1}'\gamma} \times \cdots \\
&\qquad \qquad \qquad \qquad\qquad \qquad \times \int_{w_1'=w_2'}^1 (w_1')^{-1-\ell_1'\gamma} \int_{w_0=w_1'\vee \frac{1}{k}}^1 w_0^{2\gamma-2-\ell_0\gamma} \\
&\quad \times \int_{w_q=0}^{s_\wedge(u,k)} w_q^{-(1+\ell_q)\gamma} \int_{w_{q-1}=w_q}^1 w_{q-1}^{-1-\ell_{q-1}\gamma} \times \cdots\times  \int_{w_1=w_2}^1 w_1^{-1-\ell_1\gamma}. 
\end{align*}
Since $\gamma<1/(2\ell)$, it follows that 
$$
\int_{w_q=0}^{s_\wedge(u,k)} w_q^{-(1+\ell_q)\gamma} \int_{w_{q-1}=w_q}^1 w_{q-1}^{-1-\ell_{q-1}\gamma} \times\cdots\times  \int_{w_1=w_2}^1 w_1^{-1-\ell_1\gamma} \le Cs(u,k)^{1-(1+\sum_{i=1}^q\ell_i)\gamma}. 
$$
Moreover, since $\ell_0\ge1$ with $2\gamma-1-\ell_0\gamma \le \gamma-1< \frac{1}{2\ell}-1 <0$, we deduce
$
\int_{w_0=w_1'\vee \frac{1}{k}}^1 w_0^{2\gamma-2-\ell_0\gamma} \le C \Big( w_1' \vee \frac{1}{k} \Big)^{2\gamma-1-\ell_0\gamma}, 
$
so that 
\begin{align}
C_{u,k} &\le C s(u,k)^{1-(1+\sum_{i=1}^q\ell_i)\gamma} \int_{v=0}^1 v^{-(1+p')\gamma} \int_{w_e'=v}^1 (w_e')^{-1-\ell_e'\gamma} \int_{w_{e-1}'=w_e'}^1 (w_{e-1}')^{-1-\ell_{e-1}'\gamma} \times \cdots \label{e:Cuk.upper.bdd} \\
&\qquad \qquad \qquad \times  \int_{w_2'=w_3'}^1 (w_2')^{-1-\ell_2'\gamma} \int_{w_1'=w_2'}^1  (w_1')^{-1-\ell_1'\gamma} \Big( w_1' \vee \frac{1}{k} \Big)^{2\gamma-1-\ell_0\gamma}. \notag
\end{align}

We now claim that 
\begin{equation}  \label{e:bound.C.uk}
C_{u,k} \le Cs(u,k)^{1-(1+\sum_{i=1}^q\ell_i)\gamma}\cdot  k^{-(1-p'-\ell_0-\sum_{j=1}^e \ell_j')\gamma}. 
\end{equation}
To prove this, we first assume the condition in Case (I), that is, $v > 1/k$. Then, the corresponding terms on the right-hand side of \eqref{e:Cuk.upper.bdd} are given by 
\begin{align}
&Cs(u,k)^{1-(1+\sum_{i=1}^q\ell_i)\gamma} \int_{v=1/k}^1 v^{-(1+p')\gamma} \int_{w_e'=v}^1 (w_e')^{-1-\ell_e'\gamma} \int_{w_{e-1}'=w_e'}^1 (w_{e-1}')^{-1-\ell_{e-1}'\gamma} \label{e:v>1/k}\\
&\qquad \qquad \times \cdots \times \int_{w_2'=w_3'}^1 (w_2')^{-1-\ell_2'\gamma} \int_{w_1'=w_2'}^1 (w_1')^{2\gamma-2-(\ell_0+\ell_1')\gamma}.   \notag 
\end{align}
Since $\gamma<1/(2\ell)$ and $\ell_0\ge1$, we have $2\gamma-1-(\ell_0+\ell_1')\gamma<0$, and thus, 
$$
\int_{w_1'=w_2'}^1 (w_1')^{2\gamma-2-(\ell_0+\ell_1')\gamma} \le C (w_2')^{2\gamma-1-(\ell_0+\ell_1')\gamma}. 
$$
Substituting this bound and integrating the resulting expression over $(w_2', \dots, w_d')$, we find that \eqref{e:v>1/k} is upper bounded by  the desired quantity
\begin{align*}
&Cs(u,k)^{1-(1+\sum_{i=1}^q\ell_i)\gamma} \int_{1/k}^1 v^{-1+(1-p'-\ell_0-\sum_{j=1}^e \ell_j')\gamma} \dif v 
\le Cs(u,k)^{1-(1+\sum_{i=1}^q\ell_i)\gamma} \cdot k^{-(1-p'-\ell_0-\sum_{j=1}^e \ell_j')\gamma}. 
\end{align*}
We next turn  to Case (III)  and assume  that $w_{s-1}' >1/k$ and $w_s'\le 1/k$ for some $s=2,\dots,e$. Then, the corresponding terms on the right-hand side of \eqref{e:Cuk.upper.bdd} can be bounded as follows: 
\begin{align*}
&C s(u,k)^{1-(1+\sum_{i=1}^q\ell_i)\gamma} \int_{v=0}^{1/k} v^{-(1+p')\gamma} \int_{w_e'=v}^{1/k} (w_e')^{-1-\ell_e'\gamma} \int_{w_{e-1}'=w_e'}^{1/k} (w_{e-1}')^{-1-\ell_{e-1}'\gamma}\\
&\qquad \times  \cdots\times  \int_{w_{s}'=w_{s+1}'}^{1/k} (w_{s}')^{-1-\ell_{s}'\gamma}  \int_{w_{s-1}'=1/k}^1 (w_{s-1}')^{-1-\ell_{s-1}'\gamma}  \int_{w_{s-2}'=w_{s-1}'}^1 (w_{s-2}')^{-1-\ell_{s-2}'\gamma}\\
&\qquad \times  \cdots  \times \int_{w_{2}'= w_{3}'}^1 (w_{2}')^{-1-\ell_{2}'\gamma}  \int_{w_1'=w_2'}^1  (w_1')^{2\gamma-2-(\ell_0+\ell_1')\gamma} \\
&\le C s(u,k)^{1-(1+\sum_{i=1}^q\ell_i)\gamma} \cdot k^{-(2\gamma-1-\ell_0\gamma-\sum_{j=1}^{s-1} \ell_j'\gamma)}  \\
&\qquad \times \int_{v=0}^{1/k} v^{-(1+p')\gamma} \int_{w_e'=v}^{1/k} (w_e')^{-1-\ell_e'\gamma} \int_{w_{e-1}'=w_e'}^{1/k} (w_{e-1}')^{-1-\ell_{e-1}'\gamma}\times \cdots\times  \int_{w_{s}'=w_{s+1}'}^{1/k} (w_{s}')^{-1-\ell_{s}'\gamma} \\
&\le C s(u,k)^{1-(1+\sum_{i=1}^q\ell_i)\gamma} \cdot k^{-(2\gamma-1-\ell_0\gamma-\sum_{j=1}^{s-1} \ell_j'\gamma)}   \cdot  k^{-1+(1+p'+\sum_{j=s}^e \ell_j')\gamma} \\
&=C s(u,k)^{1-(1+\sum_{i=1}^q\ell_i)\gamma} \cdot k^{-(1-p'-\ell_0-\sum_{j=1}^e \ell_j')\gamma}. 
\end{align*}
For Cases $(II)$ and $(IV)$, the required computations are almost identical to those above, and we will omit detailed arguments. Now, \eqref{e:bound.C.uk} has been established.

It remains to integrate the right-hand side of  \eqref{e:bound.C.uk} with respect to  the remaining coordinates in \eqref{e:initial.int}. 
\begin{align*}
D_k &:= \int_{u=0}^{1} \int_{w_d=u}^{s_\wedge(u,k)}\int_{w_{d-1}=w_d}^{s_\wedge(u,k)} \dots \int_{w_{q+1}=w_{q+2}}^{s_\wedge(u,k)}  s(u,k)^{1-(1+\sum_{i=1}^q\ell_i)\gamma} \cdot  k^{-(1-p'-\ell_0-\sum_{j=1}^e \ell_j')\gamma} \\
&\quad \times \Big( \prod_{i=q+1}^d w_i^{-\ell_i\gamma} \Big)u^{-p\gamma}  \int_{x=0}^1 \prod_{i=q+1}^d \int_{z_i\in [0,n]} \one \big\{ (z_{q+1}, w_{q+1}) \to \cdots \to (z_d,w_d)\to (x,u) \big\} \\
&\le C k^{-(1-p'-\ell_0-\sum_{j=1}^e \ell_j')\gamma}  \int_{u=0}^{1} s(u,k)^{1-(1+\sum_{i=1}^q\ell_i)\gamma} u^{-(p+1)\gamma} \int_{w_d=u}^{s_\wedge(u,k)} w_d^{-1-\ell_d\gamma} \\
&\quad \times \int_{w_{d-1}=w_d}^{s_\wedge(u,k)} w_{d-1}^{-1-\ell_{d-1}\gamma} \times \cdots\times  \int_{w_{q+2}=w_{q+3}}^{s_\wedge(u,k)} w_{q+2}^{-1-\ell_{q+2}\gamma} \int_{w_{q+1}=w_{q+2}}^{s_\wedge(u,k)} w_{q+1}^{-1-(\ell_{q+1}-1)\gamma}. 
\end{align*}
The innermost  integral is bounded as
\begin{equation}  \label{e:inner.bound2}
\int_{w_{q+1}=w_{q+2}}^{s_\wedge(u,k)} w_{q+1}^{-1-(\ell_{q+1}-1)\gamma} \le \begin{cases}
Cw_{q+2}^{-(\ell_{q+1}-1)\gamma} & \text{if } \ell_{q+1}\ge1, \\[5pt]
Cs(u,k)^\gamma & \text{if } \ell_{q+1}=0. 
\end{cases}
\end{equation}
If $\ell_{q+1}\ge1$, it follows from the first bound in \eqref{e:inner.bound2} that 
\begin{align*}
D_k &\le  C k^{-(1-p'-\ell_0-\sum_{j=1}^e \ell_j')\gamma}  \int_{0}^{1} s(u,k)^{1-(1+\sum_{i=1}^q\ell_i)\gamma} u^{-(p+1)\gamma} \cdot u^{-(\sum_{i=q+1}^d\ell_i-1)\gamma} \dif u \\
&= C k^{-(1-p'-\ell_0-\sum_{j=1}^e \ell_j')\gamma}  \times k^{-\frac{1-(1+\sum_{i=1}^q \ell_i)\gamma}{1-\gamma}} \int_0^{1/k} u^{-(p+\sum_{i=q+1}^d \ell_i)\gamma-\gamma \,\frac{1-(1+\sum_{i=1}^q \ell_i)\gamma}{1-\gamma}} \dif u \\
&=  C k^{-(1-p'-\ell_0-\sum_{j=1}^e \ell_j')\gamma} \times k^{-(2-(p+1+\sum_{i=1}^d \ell_i)\gamma)} \\
&= Ck^{-(2-(p+p'+\ell_0+\sum_{i=1}^d \ell_i + \sum_{j=1}^e \ell_j')\gamma)} \le C k^{-(2-2\ell\gamma)}, 
\end{align*}
where the last inequality follows from \eqref{e:ell.and.other.indices1} and \eqref{e:ell.and.other.indices2}. 
 On the other hand, even when $\ell_{q+1}=0$, by using the second bound in \eqref{e:inner.bound2} and following the same calculations as above, one can obtain the same bound $D_k\le Ck^{-(2-2\ell\gamma)}$.

Since we have completed all the necessary arguments for Cases (I)--(IV), we are left to deal with Case (V): $w_0\le 1/k$. Removing the condition $(z_0, w_0) \to (z_1,w_1)$ from \eqref{e:initial.int}, we begin with the integral concerning some of the coordinates in \eqref{e:initial.int}.
\begin{align*}
E_{u,k} &:= \int_{v=0}^{1/k} \int_{w_e'=v}^{1/k} \int_{w_{e-1}'=w_e'}^{1/k} \cdots \int_{w_1'=w_2'}^{1/k} \int_{w_0=w_1'}^{1/k} v^{-p'\gamma} \Big(\prod_{j=1}^e (w_j')^{-\ell_j'\gamma} \Big) w_0^{-\ell_0\gamma} \\
&\quad \times \int_{y=k+1}^{k+2} \prod_{j=1}^e \int_{z_j'\in [0,n]} \int_{z_0\in [0,n]} \one \big\{ (z_0,w_0)\to (z_1', w_1') \to \cdots \to (z_e', w_e') \to (y,v)\big\} \\
&\le  C \int_{v=0}^{1/k}  v^{-(1+p')\gamma}  \int_{w_e'=v}^{1/k} (w_e')^{-1-\ell_e'\gamma} \int_{w_{e-1}'=w_e'}^{1/k} (w_{e-1}')^{-1-\ell_{e-1}'\gamma} \times  \cdots  \\
&\qquad \qquad \qquad \qquad \qquad \qquad \times \int_{w_1'=w_2'}^{1/k}  (w_1')^{-1-\ell_1'\gamma} \int_{w_0=w_1'}^{1/k} w_0^{-1-(\ell_0-1)\gamma}. 
\end{align*}
Since $\ell_0\ge1$ and $\int_{w_1'}^{1/k}w_0^{-1-(\ell_0-1)\gamma} \dif w_0 \le C (w_1')^{-(\ell_0-1)\gamma}$, we have 
$$
E_{u,k}\le C \int_{0}^{1/k} v^{-(p'+\ell_0+\sum_{j=1}^e \ell_j')\gamma}\dif v = C k^{-1+(p'+\ell_0+\sum_{j=1}^e \ell_j')\gamma}. 
$$
Substituting this bound, we  calculate the integral over all the  remaining coordinates, 
\begin{align*}
F_k &:= \int_{u=0}^{1/k} \int_{w_d=u}^{1/k} \int_{w_{d-1}=w_d}^{1/k} \cdots \int_{w_1=w_2}^{1/k} k^{-1+(p'+\ell_0+\sum_{j=1}^e \ell_j')\gamma} u^{-p\gamma} \prod_{i=1}^d w_i^{-\ell_i\gamma} \\
&\quad \times \int_{x=0}^1 \prod_{i=1}^d \int_{z_i\in [0,n]} \one \big\{ (z_1,w_1)\to \cdots \to (z_d,w_d)\to (x,u) \big\}  \\
&\le C k^{-1+(p'+\ell_0+\sum_{j=1}^e \ell_j')\gamma}  \int_{u=0}^{1/k} u^{-(p+1)\gamma} \int_{w_d=u}^{1/k}  w_d^{-1-\ell_d\gamma} \int_{w_{d-1}=w_d}^{1/k} w_{d-1}^{-1-\ell_{d-1}\gamma} \times \cdots \\
&\qquad \qquad \qquad \qquad \qquad \qquad \times \int_{w_2=w_3}^{1/k} w_2^{-1-\ell_2\gamma} \int_{w_1=w_2}^{1/k} w_1^{-1-(\ell_1-1)\gamma}.  
\end{align*}
The innermost integral satisfies 
\begin{equation}  \label{e:inner.bound3}
\int_{w_1=w_2}^{1/k} w_1^{-1-(\ell_1-1)\gamma} \le \begin{cases}
Cw_2^{-(\ell_1-1)\gamma} & \text{if } \ell_1\ge1, \\[5pt]
Ck^{-\gamma} & \text{if } \ell_1=0. 
\end{cases}
\end{equation}
If $\ell_1\ge1$, then by  \eqref{e:inner.bound3}, 
\begin{align*}
F_k \le C k^{-1+(p'+\ell_0+\sum_{j=1}^e \ell_j')\gamma} \int_{0}^{1/k} u^{-(p+\sum_{i=1}^d \ell_i)\gamma} \dif u 
= Ck^{-(2-(p+p'+\ell_0+\sum_{i=1}^d \ell_i + \sum_{j=1}^e \ell_j')\gamma)}\le  C k^{-(2-2\ell\gamma)}. 
\end{align*}
 If $\ell_1=0$, then by  \eqref{e:inner.bound3} again, we have that 
\begin{align*}
F_k &\le C k^{-1+(p'-1+\ell_0+\sum_{j=1}^e \ell_j')\gamma} \int_{0}^{1/k} u^{-(p+1+\sum_{i=1}^d \ell_i)\gamma} \dif u \\
&= Ck^{-(2-(p+p'+\ell_0+\sum_{i=1}^d \ell_i + \sum_{j=1}^e \ell_j')\gamma)} \le C k^{-(2-2\ell\gamma)}. 
\end{align*}

\noindent \underline{\textit{Proof of $(ii)$}}: Using the Mecke formula for Poisson point processes, one can write 
\begin{equation}   \label{e:Var(T1)}
\ms{Var}(T_1^{(n)}) = \int_0^1 \E \big[ \inDt(u)^2 \big] \dif u + \int_{[0,1]^4} \ms{Cov} \big( \inDt(x,u), \inDt(y,v) \big) \dif x \dif u \dif y \dif v + C_1^{(n)},
\end{equation}
where $C_1^{(n)}$ is a positive remainder term  corresponding to the integral of the expectation of the count of subtrees rooted at $(x,u)$ (resp.~$(y,v)$) that contain a node $(y,v)$ (resp.~$(x,u)$). A more detailed expression for $C_1^{(n)}$, along with its analysis, can be found in \cite[Lemma 10]{hirsch:juhasz:2023}. 
By the Harris-FKG inequality, the covariance term above is non-negative, and, hence, 
$$
\ms{Var}(T_1^{(n)}) \ge \int_0^1 \E \big[ \inDt(u)^2 \big] \dif u\ge \int_0^1  \mu_{n,\ms{T}} (u)^2 \dif u. 
$$
By appealing to Fatou's lemma, 
$
\liminf_{n\to\infty} \ms{Var}(T_1^{(n)}) \ge  \int_0^1  \mu_{\ms{T}}(u)^2 \dif u>0. 
$
By an additional application of the Harris-FKG inequality, as well as the fact that $T_1 \stackrel{d}{=} \cdots \stackrel{d}{=} T_n^{(n)}$, we have $\ms{Var}(\mc T_{n, T}) \ge n \ms{Var}(T_1^{(n)})$, and now, the proof of $(ii)$ has been completed. 
\end{proof}

\begin{proof}[Proof of Theorem \ref{t:CLT.sub-tree}]
Our proof adopts an approach  similar to Theorems 4.1 and 4.8 in \cite{oliveira:2012}. We first regroup  $\bar T_i^{(n)} = T_i^{(n)} - \E \big[T_i^{(n)}  \big]$, $i=1,\dots, n$, into multiple blocks of size $\ell_n$, where $\ell_n\to\infty$ and $\ell_n=o(n)$ as $n\to\infty$. Then, $m_n:=\lfloor n/\ell_n \rfloor$ represents, asymptotically, the number of blocks (here $\lfloor z \rfloor$ is the largest integer that does not exceed $z$). We define  $s_n:=\sqrt {\ms{Var}(\mc T_{n, T})}$ and 
$$
Y_{j,\ell_n}^{(n)} := \sum_{i=(j-1)\ell_n+1}^{j\ell_n} \bar T_i^{(n)}, \ \ j=1,\dots,m_n, \ \ \ Y_{m_n+1,\ell_n}^{(n)} := \sum_{i=m_n\ell_n+1}^{n} \bar T_i^{(n)}. 
$$

We first claim that as $n\to\infty$, 
$$
\E \Big[ \exp \Big\{i\frac{u}{s_n}\big(\mc T_{n, T}-\E[\mc T_{n, T}]\big)\Big\} \Big] \sim  \prod_{j=1}^{m_n} \E \Big[ \exp\Big\{ i\frac{u}{s_n} Y_{j,\ell_n}^{(n)} \Big\} \Big], \ \ \ u \in \R. 
$$
For the proof, it is sufficient to show that for any $u\in \R$, 
\begin{align}
&\bigg| \, \E \Big[ \exp \Big\{i\frac{u}{s_n}\big(\mc T_{n, T}-\E[\mc T_{n, T}]\big)\Big\} \Big] - \E \Big[ \exp \Big\{i\frac{u}{s_n}\sum_{j=1}^{m_n}Y_{j,\ell_n}^{(n)}\Big\} \Big]  \,   \bigg| \to 0, \ \ \ n\to\infty, \label{e:ch.f.conv1}\\
&\bigg| \, \E \Big[ \exp \Big\{i\frac{u}{s_n}\sum_{j=1}^{m_n}Y_{j,\ell_n}^{(n)}\Big\} \Big] - \prod_{j=1}^{m_n} \E \Big[ \exp\Big\{ i\frac{u}{s_n} Y_{j,\ell_n}^{(n)} \Big\} \Big] \,   \bigg| \to 0, \ \ \ n\to\infty, \label{e:ch.f.conv2}
\end{align}
Since $|e^{is}-e^{it}| \le |s-t|$ for $s,t\in \R$, the Cauchy-Schwarz inequality and Lemma \ref{l:cov.var.bdd.tree} $(ii)$ give that 
$$
\bigg| \, \E \Big[ \exp \Big\{i\frac{u}{s_n}\big(\mc T_{n, T}-\E[\mc T_{n, T}]\big)\Big\} \Big] - \E \Big[ \exp \Big\{i\frac{u}{s_n}\sum_{j=1}^{m_n}Y_{j,\ell_n}^{(n)}\Big\} \Big]  \,   \bigg|  \le C|u| \Big\{ n^{-1} \E \big[ (Y_{m_n+1, \ell_n}^{(n)})^2 \big]  \Big\}^{1/2}. 
$$
By Lemma \ref{l:cov.var.bdd.tree} $(i)$, we see  that $\E \big[ (Y_{m_n+1, \ell_n}^{(n)})^2 \big] \le C \ell_n(1+ u_n(1)) \le C\ell_n$, and thus, \eqref{e:ch.f.conv1} converges to $0$ as $n\to\infty$. 

For the proof of \eqref{e:ch.f.conv2}, we employ Newman's inequality (see  Theorem \ref{t:newman.thm} in the Appendix \ref{app:C}), to ensure  that \eqref{e:ch.f.conv2} can be upper bounded by 
\begin{align*}
\frac{u^2}{s_n^2}\, \sum_{j=1}^{m_n-1}\sum_{k=j+1}^{m_n} \ms{Cov} (Y_{j,\ell_n}^{(n)}, \, Y_{k,\ell_n}^{(n)}) &\le \frac{Cu^2 m_n}{s_n^2}\, \sum_{p=1}^{\ell_n} \sum_{q=\ell_n+1}^n \ms{Cov} (T_p^{(n)}, T_q^{(n)}) \\
&\le  \frac{Cu^2m_n\ell_n}{s_n^2}\cdot \frac{1}{\ell_n}\sum_{p=1}^{\ell_n} u_n(p) \le  \frac{C}{\ell_n}\sum_{p=1}^{\ell_n} p^{-(1-2\ell\gamma)} \to 0, \ \ \text{as } n\to\infty. 
\end{align*}
At the last inequality, we have used Lemma \ref{l:cov.var.bdd.tree} $(ii)$ and \eqref{e:u.nk.bdd}. 

Now, the  proof of Theorem \ref{t:CLT.sub-tree} will be completed if one can demonstrate that 
$$
\prod_{j=1}^{m_n} \E \Big[\exp\Big\{ i\frac{u}{s_n}Y_{j,\ell_n}^{(n)} \Big\} \Big] \to e^{-u^2/2}, \ \ \text{for all } u \in \R; 
$$
equivalently, $s_n^{-1} \sum_{j=1}^{m_n} \widetilde Y_{j,\ell_n}^{(n)} \stackrel{d}{\to} \mathcal N(0,1)$, where $\big(\widetilde Y_{j,\ell_n}^{(n)}\big)_{j=1}^{m_n}$ is an i.i.d.~copy of $\big(Y_{j,\ell_n}^{(n)}\big)_{j=1}^{m_n}$. For this purpose, we note that 
$
s_n^2 \sim \ms{Var}\Big( \sum_{j=1}^{m_n} \widetilde Y_{j,\ell_n}^{(n)} \Big) = m_n \ms{Var}(Y_{1,\ell_n}^{(n)}), \ \ \text{as } n\to\infty. 
$
Thus, by  Slutsky's theorem, it suffices to verify that as $n\to\infty$, 
\begin{equation}  \label{e:goal.clt}
\frac{1}{\sqrt{m_n \ms{Var}(Y_{1,\ell_n}^{(n)})}} \sum_{j=1}^{m_n} \widetilde Y_{j,\ell_n}^{(n)} \stackrel{d}{\to} \mathcal N(0,1). 
\end{equation}
According to the Lindeberg-Feller central limit theorem (see Theorem 2 on p.~334 in \cite{shiryaev:1996}), it is sufficient to show   that for every $\epsilon>0$, 
\begin{equation}  \label{e:Lindeberg.CLT}
\frac{1}{m_n\ms{Var}(Y_{1,\ell_n}^{(n)})}\, \sum_{j=1}^{m_n} \E \bigg[ (Y_{j,\ell_n}^{(n)})^2 \, \one \Big\{ |Y_{j,\ell_n}^{(n)}| \ge \epsilon \sqrt{m_n\text{Var}(Y_{1,\ell_n}^{(n)})}  \Big\}\bigg] \to 0, \ \ \ n\to\infty;
\end{equation}
equivalently, 
$\frac{1}{s_n^2}\,  \sum_{j=1}^{m_n} \E \bigg[ (Y_{j,\ell_n}^{(n)})^2 \, \one \Big\{ |Y_{j,\ell_n}^{(n)}| \ge \epsilon s_n \Big\}\bigg] \to 0$,  $n\to\infty. $
We observe that we have the bound $(Y_{j,\ell_n}^{(n)})^2 \le \ell_n \sum_{i=(j-1)\ell_n+1}^{j\ell_n}(\bar T_i^{(n)})^2$. Thus, together with the union bound, \eqref{e:Lindeberg.CLT} is upper bounded by 
\begin{align*}
&\frac{1}{s_n^2} \sum_{j=1}^{m_n}\ell_n  \sum_{i=(j-1)\ell_n+1}^{j\ell_n} \sum_{k=(j-1)\ell_n+1}^{j\ell_n} \E \Big[ (\bar T_i^{(n)})^2 \one \big\{ |\bar T_i^{(n)}| \ge \epsilon s_n/\ell_n \big\} \Big] =\frac{m_n\ell_n^3}{s_n^2} \E \Big[ (\bar T_1^{(n)})^2 \one \big\{ |\bar T_1^{(n)}| \ge \epsilon s_n/\ell_n \big\} \Big]. 
\end{align*}
Here, we assume temporarily that there exists $\delta\in (0,1)$ such that 
$
\sup_{n} \E \big[ (\bar T_1^{(n)})^{2+\delta} \big] <\infty$.
Then, H\"older's inequality yields that 
$$
	\E \Big[ (\bar T_1^{(n)})^2 \one \big\{ |\bar T_1^{(n)}| \ge \epsilon s_n/\ell_n \big\} \Big] \le \Big\{  \E \big[ (\bar T_1^{(n)})^{2+\delta} \big]  \Big\}^{\f2{2+\delta}} \P \Big( |\bar T_1^{(n)}| \ge \frac{\epsilon s_n}{\ell_n} \Big)^{\delta/(2+\delta)} \le C \P \Big( |\bar T_1^{(n)}| \ge \frac{\epsilon s_n}{\ell_n} \Big)^{\f\delta{2+\delta}}. 
$$
By Markov's inequality, 
$
\P \Big( |\bar T_1^{(n)}| \ge \frac{\epsilon s_n}{\ell_n} \Big) \le  \frac{2\ell_n}{\epsilon s_n}\int_0^1 \mu^{(n)}(u)\dif u \le   \frac{C\ell_n}{s_n}. 
$
Substituting them, we get that
$$
\frac{1}{s_n^2}\,  \sum_{j=1}^{m_n} \E \bigg[ (Y_{j,\ell_n}^{(n)})^2 \, \one \Big\{ |Y_{j,\ell_n}^{(n)}| \ge \epsilon s_n \Big\}\bigg]  \le C\, \frac{m_n\ell_n^{3+\frac{\delta}{2+\delta}}}{s_n^{2+\frac{\delta}{2+\delta}}} \le C \Big( \frac{\ell_n^{4+3\delta}}{n^{\delta/2}} \Big)^{1/(2+\delta)}, 
$$
where the last inequality follows from Lemma \ref{l:cov.var.bdd.tree} $(ii)$. By taking $\ell_n\to \infty$ to ensure that $\ell_n = o\big( n^{\delta/(2(4+3\delta))} \big)$, the final term above tends to $0$ as $n\to\infty$. 

To conclude the proof, it remains to verify $\sup_{n} \E \big[ (\bar T_1^{(n)})^{2+\delta} \big] <\infty$. Here, it is enough to show that $\sup_{n} \E \big[ ( T_1^{(n)})^{2+\delta} \big] <\infty$. By H\"older's inequality, 
\begin{align*}
\E \big[ (T_1^{(n)})^{2+\delta} \big] &\le \E \Big[ \mathcal P\big(  [0,1]^2\big)^{2+\delta} \max_{P\in \mathcal P\cap [0,1]^2} \inDt(P)^{2+\delta}    \Big] \\
&\le \bigg\{ \E\Big[\mathcal P\big(  [0,1]^2\big)^{(2+\delta)(1+\epsilon_0)/\epsilon_0} \Big]\bigg\}^{\epsilon_0/(1+\epsilon_0)} \bigg\{ \E\Big[ \max_{P\in \mathcal P\cap [0,1]^2} \inDt(P)^{(2+\delta)(1+\epsilon_0)} \Big]   \bigg\}^{1/(1+\epsilon_0)}, 
\end{align*}
where $\epsilon_0>0$ satisfies $(2+\delta)(1+\epsilon_0)<3$. Since the $(2+\delta)(1+\epsilon_0)/\epsilon_0$th moment of $\mathcal P\big( [0,1]^2 \big)$ is clearly finite, we only demonstrate that 
$
\sup_{n} \E \Big[  \max_{P\in \mathcal P\cap [0,1]^2} \inDt(P)^{(2+\delta)(1+\epsilon_0)} \Big]<\infty. 
$
By applying again H\"older's inequality and Proposition \ref{p:exp.var.sub-tree} $(iii)$, 
$$
 \E \Big[  \max_{P\in \mathcal P\cap [0,1]^2} \inDt(P)^{(2+\delta)(1+\epsilon_0)} \Big] \le \E\Big[ \sum_{P\in \mathcal P\cap [0,1]^2}\hspace{-.4cm} \inDt(P)^{(2+\delta)(1+\epsilon_0)} \Big] 
 = \int_0^1 \E \big[ \inDt(u)^{(2+\delta)(1+\epsilon_0)} \big] \dif u.$$
	 Now, 
	 $\E \big[ \inDt(u)^{(2+\delta)(1+\epsilon_0)} \big] 
 \le\Big\{ \E \big[ \inDt(u)^3 \big] \Big\}^{(2+\delta)(1+\epsilon_0)/3}  \le C (u^{-3\ell\gamma-\eta})^{(2+\delta)(1+\epsilon_0)/3}$. 
It can be assured from the assumption $0<\gamma<1/(2\ell)$ that the integral over the final term becomes finite if we set  $\delta, \epsilon_0$, and $\eta$ sufficiently small. 
\end{proof}

\newpage
\begin{appendix}
\section{Proof of Lemma \ref{lem:cov2}}    \label{app:A}
%
%
To prove Lemma \ref{lem:cov2}, we establish an integral bound that will be used frequently in the following.
\bel[Integral bound]
\label{lem:int}
Let $0<\g < 1/2$ and $a \in (3- \g - 1/\g, 2 -\g)$. Then,    
$$\int_0^1 w^{-a} \big(w \wedge (w^{1- \g} n_1)^{-1/\g}\big)^{1 - \g} \d w \in O(n^{a + \g -2}).$$
\enl
\bep
First, if $w \le n_1^{-1}$, then we obtain an upper bound of order $O(n^{a + \g - 2})$. If $w \ge n_1^{-1}$, we again obtain a bound of order  $n^{-(1 - \g)/\g}\int_{n_1^{-1}}^1 w^{-(1 - \g)^2/\g - a} \d w  \in O( n^{ a +\g - 2}).$
\enp

Now, we are ready to prove Lemma \ref{lem:cov2}.

%
%
\bep[Proof Lemma \ref{lem:cov2} -- part $(i)$]
Let $(Z, W)$ and $(Z', W')$ be the highest points of the $k$-clique and the $\ell$-clique, respectively. We distinguish two cases depending on whether or not one of the highest points is in the down-neighborhood of the other one. Note that it could also be that $(Z, W) = q$ or $(Z', W') = (0, u)$. However, we do not consider these cases  since the arguments are easier than in the other case.

\emph{Case 1.} We assume without loss of generality that $(Z', W') \in N^\da((Z, W))$, since the other  case $(Z, W) \in N^\da((Z', W'))$ is analogous. Note that any clique point is in the down-neighborhood of $(Z, W)$ or $(Z', W')$. Since the number of points in the down-neighborhood is Poisson distributed with parameter of constant order, this provides a constant contribution. Hence, we need to bound the expected number of configurations  $$(0, U) \leftarrow (Z, W) \to (Z', W') \to q=:(Y, V),$$ such that 
at least one of these edges has a length of at least $n_1$. 

We discuss the following  three cases separately. All the above computations below  are based on the same techniques as in   \eqref{eq:neighb-technique} and  \eqref{eq:exp1} and we do not put all details.
First, 	consider the case \emph{$|Z| \ge n_1$}. Then, using the estimate from Lemma \ref{lem:int} with $a = 1 -\g$,
\begin{align*}
	&	\int_0^1\int_0^w (u^{-\g}w^{\g -1} - n_1)_+ \d u \d w  \le C\int_0^1 w^{\g - 1}(w\wedge (w^{1 - \g}n_1)^{-1/\g})^{1 - \g}  \d w \in O(n^{-1}).
\end{align*}
Second, consider the case where  \emph{$|Z' - Z| \ge n_1$}. Then, we need to bound the expression
\begin{align*}
	&	 \int_0^1 \int_0^w  \big(w^{\g - 1} (w')^{-\g} - n_1\big)_+ \d w'\d w
	= C\int_0^1  w^{\g - 1}(w \wedge (w^{1-\g}n_1)^{-1/\g})^{ 1- \g}\d w.
\end{align*}
As above, this is of order $O(n^{-1})$.

Finally, consider the case \emph{$|Z' - Y| \ge n_1$}. Then, we need to bound the expression
\begin{align*}
	&	 \int_0^1 \int_0^w \int_0^{w'} \int_0^wu^{-\g}w^{2\g - 2}(w')^{-\g} \big((w')^{\g - 1} v^{-\g} - n_1\big)_+ \d u\d v\d w'\d w\\
	&= C\int_0^1 \int_0^w \int_0^{w'} w^{\g - 1}(w')^{-\g} \big((w')^{\g - 1} v^{-\g} - n_1\big)_+ \d v\d w'\d w\\
	&= C\int_0^1 (w')^{-1}  (w' \wedge ((w')^{1- \g}n_1)^{-1/\g})^{1-\g} \d w'.
\end{align*}
Applying Lemma \ref{lem:int} with $a=1$ shows that this expression is in $O(n^{-(1 - \g)})$.
\smallskip

\emph{Case 2.} The two highest points $(Z, W)$ and $(Z', W')$ are different and none of them belongs  to the other clique. However, by the covariance property, we know that there exists at least one common point, denoted  $(Z'', W'')$, between the two cliques.  Hence, we need to bound the expected number of configurations 
\begin{align}
	\label{eq:uyv}
	(0, U) \leftarrow (Z, W) \to (Z'', W'') \leftarrow (Z', W') \to (Y, V), 
\end{align}
such that at least one of these edges has a length of at least $n_1$.  

We discuss the following four cases separately.
First, consider the case where $|Z| \ge n_1$. Then, by Lemma \ref{lem:int} with $a = \g$,
\begin{align*}
	\int_0^1\int_0^w \ w^{-\g}(u^{-\g}w^{\g - 1} - n_1)_+ \d u \d w  = C\int_0^1 w^{-\g}(w \wedge(w^{1 - \g}n_1)^{-1/\g})^{1 - \g} \d w \in O(n^{-2(1- \g)}).
\end{align*}
Second, consider the case where $|Z' - Y| \ge n_1$. Then, by Lemma \ref{lem:int} with $a = \g$,
\begin{align*}
\int_0^1\int_0^{w'}(w')^{-\g}(v^{-\g}(w')^{\g - 1} - n_1)_+ \d w' \d v  = C\int_0^1 (w')^{-\g}(w' \wedge((w')^{1 - \g}n_1)^{-1/\g})^{1 - \g} \d w' \in O(n^{-2(1 - \g)}).
\end{align*}

Third, consider the case where $|Z - Z''| \ge n_1$. Then, by Lemma \ref{lem:int} with $a = \g$,
$$\int_0^1\int_0^w (w'')^{-\g}(w^{\g - 1}(w'')^{-\g} - n_1)_+ \d w'' \d w = C\int_0^1 w^{\g - 1} (w\wedge(w^{1 - \g}n_1)^{-1/\g})^{1 - 2\g} \d w \in O(n^{-(1 - \g)})$$
Finally, consider the case where $|Z' - Z''| \ge n_1$. Then, we need to bound the integral
$$\int_0^1\int_0^w (w'')^{-\g}((w')^{\g - 1}(w'')^{-\g} - n_1)_+ \d w'' \d w'.$$
Hence, the calculations are identical to the second case.
\enp

Now, we prove part $(ii)$. Similarly to part $(i)$, this involves several case distinctions. 
\bep[Proof Lemma \ref{lem:cov2} -- part $(ii)$] As in the proof of part $(i)$, to avoid redundancies, we do not treat explicitly the cases $(Z, W) = q$ or $(Z', W') = (0, u)$. We first consider again the configuration of the form
	$$(0, U) \leftarrow (Z, W) \to (Z'', W'') \leftarrow (Z', W') \to (Y, V)$$
	as in \eqref{eq:uyv}. If one of the edges at $(0, U)$ or $(Y, V)$ has length more  than $n_1$, then we conclude as in part $(i)$. 

Hence, it remains to consider the case where one of the edges $(Z, W) \to (Z'', W'')$ or $(Z', W') \to (Z'', W'')$ is at least of length $n_1$. We only deal with the  former case, since the argument for the latter is similar. If this edge goes to $(Z', W')$ or  $(Z'', W'')$, then  we can again argue as in part $(i)$. Therefore, it remains to deal with the case where the long edge goes to a different vertex.

First, the expected number of edges of the form $(z,w)\to(\bar Z, \bar W)$ of length at least $n_1$ is of order at most
\begin{align}
	\label{eq:cov21}
	Cw^{\g - 1}(w\wedge (w^{1 - \g}n_1)^{-1/\g})^{1 - \g}.
\end{align}
Second, the expected number of  configurations 
$(z, w) \to (Z'', W'')\leftarrow (Z', W')$
is of order 
\begin{align}
	\label{eq:cov22}
	C w^{\g - 1}\int_0^w (w'')^{-2\g} \d w \in O\big(w^{-\g}\big).
\end{align}
Hence, combining \eqref{eq:cov21} and \eqref{eq:cov22} yields the upper bound 
$C\int_0^1 w^{-1}(w\wedge (w^{1 - \g}n_1)^{-1/\g})^{1 - \g} \d w$. By Lemma \ref{lem:int} with $a = 1$, this expression is of order $O(n^{-(1 - \g)})$, thereby concluding the proof.
\enp
%
%
\section{Proof of Proposition  \ref{p:exp.var.sub-tree}}   \label{app:B}
\label{sec:rem}
\begin{proof}[Proof of Proposition \ref{p:exp.var.sub-tree}]
The proof of  part $(i)$  is provided in  \cite[Proposition 2.1]{ht24}. For the proof of part $(ii)$, given a directed tree $\ms{T}$ on $m$ vertices with root $\ms{r}$, we start by expressing   $\inDt(u)$  as follows. 
\begin{equation}  \label{e:expression.Din}
\inDt(u) = \sum_{(P_1,\dots,P_{m-1})\in (\Pn)_{\neq}^{m-1}} h\big( (0,u), P_1,\dots, P_{m-1}\big), 
\end{equation}
where $(\Pn)_{\neq}^{m-1} =\big\{ (P_1,\dots,P_{m-1})\in \Pn^{m-1}: P_i \neq P_j \text{ for } i \neq j \big\}$ and  $h$ is an indicator function for the event that there exists a graph homomorphism from $\ms{T}$ to an induced graph defined on  $\big\{ (0,u), P_1,\dots, P_{m-1} \big\}$, with the root $\ms{r}$  mapped to $(0,u)$. 

By the Mecke formula for Poisson point processes, 
\begin{align}
\begin{split}  \label{e:var.mecke}
\ms{Var}\big( \inDt(u)\big) &=\E \bigg[ \sum_{(P_1, \dots, P_{m-1})\in (\Pn)_{\neq}^{m-1}}\sum_{(Q_1,\dots,Q_{m-1})\in (\Pn)_{\neq}^{m-1}} \hspace{-10pt}h\big( (0,u),P_1,\dots,P_{m-1} \big)\, \\
&\qquad \qquad\qquad\qquad\qquad\qquad\qquad\qquad\times h\big( (0,u),Q_1,\dots,Q_{m-1} \big)  \bigg] - \mu_{n,\ms{T}}(u)^2 \\
&=\mu_{n,\ms{T}}(u) + \E \bigg[ \sum_{(P_1, \dots, P_{m-1})\in (\Pn)_{\neq}^{m-1}} \hspace{-10pt}\sum_{\substack{(Q_1,\dots,Q_{m-1})\in (\Pn)_{\neq}^{m-1}, \\ 1\le |(P_1,\dots,P_{m-1})\cap (Q_1,\dots,Q_{m-1})|\le m-2}} \hspace{-20pt}h\big( (0,u),P_1,\dots,P_{m-1} \big)\, \\
&\qquad \qquad\qquad\qquad\qquad\qquad\qquad\qquad\qquad \qquad \times h\big( (0,u),Q_1,\dots,Q_{m-1} \big)  \bigg]. 
\end{split}
\end{align}
The asymptotics of $\mu_{n,\ms{T}}(u)$ is already given in part $(i)$, so we focus on the second term above. As discussed in the proof of Proposition 2.2 in \cite{ht24}, it suffices to consider the case where $|(P_1,\dots,P_{m-1})\cap (Q_1, \dots, Q_{m-1})|=1$. Let $\ms T_i$, $i=1,2$, be a copy of $\ms T$, and select one node from each of $\ms T_1$ and $\ms T_2$, identifying the two selected nodes together. For the calculation of the second term in \eqref{e:var.mecke}, we focus on the configuration of nodes and edges defined by 
\begin{align}
\begin{split}  \label{e:zi.wi.zj'.wj'}
&(z_0, w_0) \to (z_1,w_1) \to \cdots \to (z_d,w_d) \to (0,u), \\
&(z_0, w_0) \to (z_1',w_1') \to \cdots \to (z_e',w_e') \to (0,u), 
\end{split}
\end{align}
for some $d, e \geq 0$, where $(z_0,w_0)$ represents a common node between $(P_1,\dots,P_{m-1})$ and $(Q_1,\dots, Q_{m-1})$. 
In the special case of $d=0$ or $e=0$, the point $(z_0,w_0)$ directly connects to $(0,u)$. 
If $d=e=0$, then \eqref{e:zi.wi.zj'.wj'} does not imply a cycle, whereas a cycle occurs only when either $d$ or $e$ is positive. Since the latter case is more fundamental, we assume  that $e\ge1$ in the below. 
To clarify the configuration in \eqref{e:zi.wi.zj'.wj'}, we assume, without loss of generality, that the path in the first line of \eqref{e:zi.wi.zj'.wj'} is taken from $\ms T_1$, while the path in the second line is taken from $\ms T_2$. Furthermore, assume there are $\ell_0$ leaves with no further children that are connected to $(z_0,w_0)$ by the paths, all consisting of nodes with marks higher than $(z_0,w_0)$. Similarly, for each $i=1,\dots,d$ (resp.~$j=1,\dots,e$), there are $\ell_i$ (resp.~$\ell_j'$) leaves with no further children that are connected to $(z_i,w_i)$ (resp.~$(z_j',w_j')$) through paths that do not include $(z_{i-1}, w_{i-1})$ (resp.~$(z_{j-1}', w_{j-1}')$) and consist entirely of nodes with marks higher than $(z_i,w_i)$ (resp.~$(z_j',w_j')$). 
Additionally, let $p$ denote the number of leaves with no further children that are connected to $(0,u)$ through paths containing neither $(z_d,w_d)$ nor $(z_e',w_e')$, where all nodes on these paths have marks higher than $(0,u)$.

Under this setup, there are three possible cases regarding the relation between $\ell$ and the other integers $p$, $\ell_i$, and $\ell_j$. 
First, if the node $(z_0,w_0)$ is not taken from the leaves of either $\ms{T}_1$ or $\ms{T}_2$, then 
\begin{equation}  \label{e:ell.others1}
2\ell = p + \ell_0 + \sum_{i=1}^d \ell_i + \sum_{j=1}^e \ell_j'. 
\end{equation}
Second, if $(z_0,w_0)$ corresponds to a leaf in only one of $\ms T_1$ or $\ms T_2$ but not both, then 
\begin{equation}  \label{e:ell.others2}
2\ell = p + \ell_0 + \sum_{i=1}^d \ell_i + \sum_{j=1}^e \ell_j' + 1. 
\end{equation}
Finally, if $(z_0,w_0)$ corresponds to a leaf in both $\ms{T}_1$ and $\ms{T}_2$, then $\ell_0=0$ and 
\begin{equation}  \label{e:ell.others3}
\ell = \frac{p}{2} + \sum_{i=1}^d \ell_i +1  = \frac{p}{2} + \sum_{j=1}^e \ell_j' +1. 
\end{equation}
It is important to note that if $\ell_0 \geq 2$, then either \eqref{e:ell.others1} or \eqref{e:ell.others2} applies. If $\ell_0 = 1$, \eqref{e:ell.others2} always applies. Lastly, \eqref{e:ell.others3} applies only when $\ell_0 = 0$. 

Before continuing, we want to point out how to integrate out the indicator for the presence of an edge over spatial coordinates under the toroidal metric: for $(x,u), (y,v) \in \T_n$ with $0 \le u \le v \le 1$, 
\begin{align}
\label{e:torus.bdd}
\int_{[0,n]} \one \big\{ (y,v) \to (x,u) \big\} \dif x &= \int_{[0,n]} \one \big\{  \ms{dist}_n (x,y) \le \beta u^{-\gamma}v^{\gamma-1}\big\} \dif x    \le 2\beta u^{-\gamma}v^{\gamma-1}. 
\end{align}
Clearly, the same bound is obtained even when $\dif x$ above is replaced with $\dif y$. 

Returning to the second term in \eqref{e:var.mecke}, applying the Mecke formula, and integrating out all variables except those in \eqref{e:zi.wi.zj'.wj'}, it remains to bound the expression 
\begin{align}
\begin{split}  \label{e:A.u}
A_u &:= \int_{w_0=u}^1 \int_{w_d=u}^{w_0} \int_{w_{d-1}=w_d}^{w_0} \cdots \int_{w_1=w_2} ^{w_0} \int_{w_e'=u}^{w_0} \int_{w_{e-1}'=w_e'}^{w_0} \cdots \int_{w_1'=w_2'}^{w_0}  u^{-p\gamma} \Big( \prod_{i=0}^d w_i^{-\ell_i\gamma}\Big)\Big( \prod_{j=1}^e (w_j')^{-\ell_j'\gamma}\Big) \\
&\quad \times \int_{z_0\in [0,n]} \prod_{i=1}^d\int_{z_i\in [0,n]} \prod_{j=1}^e \int_{z_j'\in [0,n]}  \one \big\{ (z_0, w_0) \to (z_1,w_1) \to \cdots \to (z_d,w_d) \to (0,u) \big\} \\
&\qquad \qquad \qquad \qquad\qquad \qquad\qquad \qquad \quad \times \one \big\{ (z_0, w_0) \to (z_1',w_1') \to \cdots \to (z_e',w_e') \to (0,u) \big\}. 
\end{split}
\end{align}
More precisely, it may be necessary to include additional logarithmic factors in the integrands in \eqref{e:A.u}. For instance, $u^{-p\gamma} (\log u^{-1})^s$ could appear instead of $u^{-p\gamma}$ for some $s=1,2,\dots$. In such cases, we bound $(\log u^{-1})^s$ by $Cu^{-\eta}$ for a sufficiently small $\eta > 0$. To simplify the presentation, however, these logarithmic factors are omitted throughout the remainder of the proof. Note further that $A_u$ depends on $n$, owing to the integration domain $[0,n]$ (and the related toroidal metric). However, employing bounds of the form \eqref{e:torus.bdd}, the dependence on $n$ can be immediately removed. Consequently, the dependence on $n$ is suppressed in \eqref{e:A.u} for simplicity, and this convention will be applied  throughout the proof. 

By removing the condition $(z_0,w_0)\to (z_1',w_1')$ and calculating the integral  \eqref{e:A.u} sequentially over all spatial coordinates, 
\begin{align}
&\int_{z_0\in [0,n]} \prod_{i=1}^d\int_{z_i\in [0,n]} \prod_{j=1}^e \int_{z_j'\in [0,n]} \one \big\{ (z_0, w_0) \to (z_1,w_1) \to \cdots \to (z_d,w_d) \to (0,u) \big\} \label{e:remove.youngest.edge} \\
&\qquad \qquad \qquad \qquad\qquad \qquad\qquad \qquad \quad \times \one \big\{ (z_0, w_0) \to (z_1',w_1') \to \cdots \to (z_e',w_e') \to (0,u) \big\} \notag \\
&\le C u^{-2\gamma} w_0^{\gamma-1} \Big( \prod_{i=1}^d w_i^{-1} \Big) (w_1')^{\gamma-1} \Big\{ \prod_{j=2}^e (w_j')^{-1} \Big\}. \notag 
\end{align}
Substituting this bound into $A_u$, we have 
\begin{align}
\begin{split}  \label{e:Au.var.bound}
A_u &\le Cu^{-(p+2)\gamma} \int_{w_0=u}^1 w_0^{-1-(\ell_0-1)\gamma} \int_{w_d=u}^{w_0} w_d^{-1-\ell_d\gamma}   \int_{w_{d-1}=w_d}^{w_0} w_{d-1}^{-1-\ell_{d-1}\gamma}  \times \cdots \times \int_{w_1=w_2} ^{w_0} w_1^{-1-\ell_1\gamma}  \\
&\quad \times \int_{w_e'=u}^{w_0} (w_e')^{-1-\ell_e'\gamma} \int_{w_{e-1}'=w_e'}^{w_0} (w_{e-1}')^{-1-\ell_{e-1}'\gamma} \times \cdots \times \int_{w_{2}'=w_3'}^{w_0} (w_{2}')^{-1-\ell_{2}'\gamma} \int_{w_1'=w_2'}^{w_0} (w_1')^{-1-(\ell_1'-1)\gamma}. 
\end{split}
\end{align}
The innermost integral can be bounded as 
\begin{equation} \label{e:inner.bound4}
\int_{w_1'=w_2'}^{w_0} (w_1')^{-1-(\ell_1'-1)\gamma} \le \begin{cases}
C(w_2')^{-(\ell_1'-1)\gamma} & \text{if } \ell_1'\ge1, \\[5pt]
Cw_0^\gamma & \text{if } \ell_1'=0. 
\end{cases}
\end{equation}
(If $\ell_1'=1$, the upper bound takes a logarithmic form, but as noted earlier, such logarithmic terms are systematically ignored).  Suppose first that $\ell_1'\ge1$. Then, by applying the first bound in \eqref{e:inner.bound4} and sequentially integrating the resulting expression, we obtain 
\begin{equation}  \label{e:A_n.ell_0}
A_u \le C u^{-(p+1+\sum_{i=1}^d \ell_i + \sum_{j=1}^e \ell_j')\gamma} \int_{w_0=u}^1 w_0^{-1-(\ell_0-1)\gamma} \le \begin{cases}
C u^{-(p+\ell_0+\sum_{i=1}^d \ell_i + \sum_{j=1}^e \ell_j')\gamma} & \text{if } \ell_0\ge1, \\[5pt]
Cu^{-(p+1+\sum_{i=1}^d \ell_i + \sum_{j=1}^e \ell_j')\gamma} & \text{if } \ell_0=0. 
\end{cases}
\end{equation}
Recall that if $\ell_0\ge1$, either \eqref{e:ell.others1} or \eqref{e:ell.others2} holds. In both cases, the first term in \eqref{e:A_n.ell_0} is further bounded by $Cu^{-2\ell\gamma}$. If $\ell_0=0$, \eqref{e:ell.others3} always holds and the second term in \eqref{e:A_n.ell_0} is again upper bounded by $Cu^{-2\ell\gamma}$. Throughout our discussion, logarithmic terms have been systematically ignored. However, whenever they appear, they can be bounded by $Cu^{-\eta}$ for sufficiently small $\eta > 0$. From this perspective, $A_u$ is ultimately bounded by $Cu^{-2\ell\gamma-\eta}$, as required. 

Let us return to \eqref{e:Au.var.bound} and suppose next that $\ell_1'=0$. Then, we use the second bound in \eqref{e:inner.bound4}. Similarly to the last case, we conclude that 
\begin{equation}  \label{e:A_n.ell_0.2}
A_u \le C u^{-(p+2+\sum_{i=1}^d \ell_i + \sum_{j=1}^e \ell_j')\gamma} \int_{w_0=u}^1 w_0^{-1-(\ell_0-2)\gamma} \le \begin{cases}
C u^{-(p+\ell_0+\sum_{i=1}^d \ell_i + \sum_{j=1}^e \ell_j')\gamma} & \text{if } \ell_0\ge2, \\[5pt]
Cu^{-(p+2+\sum_{i=1}^d \ell_i + \sum_{j=1}^e \ell_j')\gamma} & \text{if } \ell_0 \in \{ 0,1 \}. 
\end{cases}
\end{equation}
Once again, it can be shown that, for any $\ell_0$, the final term in \eqref{e:A_n.ell_0.2} is bounded by $Cu^{-2\ell\gamma}$. More precisely, due to the possible contribution of logarithmic terms, we have $A_u \leq Cu^{-2\ell\gamma - \eta}$ for sufficiently small $\eta > 0$.
\medskip

\noindent \underline{\textit{Proof of $(iii)$}}: It follows from the Mecke formula that 
\begin{align}
\begin{split}  \label{e:Mecke.3rd.moment}
\E \big[ \inDt(u)^3 \big] &=  \mu_{n,\ms{T}}(u)^3 \\
&+ \E \bigg[ \sum_{(P_1,\dots, P_{m-1}) \in (\Pn)_{\neq}^{m-1}} \sum_{(Q_1,\dots, Q_{m-1}) \in (\Pn)_{\neq}^{m-1}} \sum_{(R_1,\dots, R_{m-1}) \in (\Pn)_{\neq}^{m-1}} \hspace{-10pt} h\big( (0,u), P_1, \dots, P_{m-1} \big) \\ 
&\qquad \qquad \qquad \qquad \times h\big( (0,u), Q_1, \dots, Q_{m-1} \big)\, h\big( (0,u), R_1, \dots, R_{m-1} \big) \, \\
&\qquad \qquad \qquad \qquad \times  \one \big\{  \text{there are at least one common points between }  (P_1,\dots, P_{m-1}), \\
&\qquad \qquad \qquad \qquad \qquad (Q_1,\dots, Q_{m-1}) \text{ and } (R_1,\dots, R_{m-1}) \big\}\bigg]. 
\end{split}
\end{align}
By Proposition \ref{p:exp.var.sub-tree} $(i)$, we observe that 
\begin{equation}  \label{e:main.3rd.moment}
\mu_{n,\ms{T}}(u)^3 \le (C u^{-\ell \gamma -\eta/3})^3 =C u^{-3\ell\gamma-\eta}. \end{equation}
For the second term in \eqref{e:Mecke.3rd.moment}, whenever a cyclic structure arises as in \eqref{e:zi.wi.zj'.wj'}, we will collapse it   by removing an edge formed by the youngest pair of nodes, such as $(z_0, w_0) \to (z_1', w_1')$ in \eqref{e:remove.youngest.edge}. As implied by  the proof of Proposition \ref{p:exp.var.sub-tree} $(ii)$, the removal of such an edge will not lead to a growth rate exceeding \eqref{e:main.3rd.moment}. Specifically, the removal of this edge introduces a new leaf, such as $(z_1, w_1)$ in \eqref{e:remove.youngest.edge}, whose mark is, however, lower than that of the common node as $(z_0, w_0)$ in \eqref{e:remove.youngest.edge}. This constraint ensures that the removal of an edge does not increase the growth rate of the second term in \eqref{e:Mecke.3rd.moment} beyond that of \eqref{e:main.3rd.moment}. 

Finally, the remaining bounds for $\mu_{\ms{T}}(u)$, $\text{Var}(\inD(u))$, and $\E \big[ \inD(u)^3 \big]$ follow directly by repeating the proofs of $(i)$–$(iii)$ and applying the monotone convergence theorem.
\end{proof}
\bigskip

\section{Newman's inequality}   \label{app:C}

Finally, we present the so-called Newman's inequality, which is one of the  main ingredients for the proof of Theorem \ref{t:CLT.sub-tree}. 

\begin{theorem}[\cite{newman:1980}]  \label{t:newman.thm}
Let $(X_n)_{n\in \N}$ be positively associated random variables. Let $\varphi_{(X_1,\dots,X_n)}(u_1,\dots,u_n) = \E\big[e^{i\sum_{j=1}^n u_j X_j}\big]$ be the characteristic function. Then, for all $n\ge1$ and $u_1,\dots, u_n\in \R$, 
$$
\Big|\, \varphi_{(X_1,\dots,X_n)}(u_1,\dots,u_n) -\prod_{j=1}^{m_n} \varphi_{X_j}(u_j)   \, \Big|\le \sum_{j=1}^{n-1}\sum_{k=j+1}^n |u_ju_k| \ms{Cov}(X_j,X_k). 
$$
\end{theorem}
\end{appendix}

\section*{Acknowledgments}
This research was partially carried out during a visit of TO to Aarhus University funded by AUFF visit grant
AUFF-E-2023-6-38. Part of this research was conducted while CH was in Paris, supported by the Programme d’invitations internationales scientifiques, campagne 2025 from Universit\'e Paris Cit\'e.
TO was partially supported by the AFOSR grant FA9550-22-1-0238 at Purdue University. CH was supported by a research grant (VIL69126) from VILLUM FONDEN.

\bibliographystyle{abbrv}
\bibliography{lit.bib}

\begin{thebibliography}{10}

\bibitem{ChungLu}
W.~Aiello, F.~Chung, and L.~Lu.
\newblock A random graph model for power law graphs.
\newblock {\em Exp. Math.}, 10:53--66, 2001.

\bibitem{barysh}
Y.~Baryshnikov and J.~E. Yukich.
\newblock Gaussian limits for random measures in geometric probability.
\newblock {\em Ann. Appl. Probab.}, 15(1A):213--253, 2005.

\bibitem{cox:grimmett:1984}
J.~T. Cox and G.~Grimmett.
\newblock Central limit theorems for associated random variables and the
  percolation model.
\newblock {\em Ann. Probab.}, 12:514--528, 1984.

\bibitem{glm2}
P.~Gracar, A.~Grauer, L.~L\"{u}chtrath, and P.~M\"{o}rters.
\newblock The age-dependent random connection model.
\newblock {\em Queueing Systems}, 93(3-4):309--331, 2019.

\bibitem{komjathy2}
P.~Gracar, M.~Heydenreich, C.~M\"{o}nch, and P.~M\"{o}rters.
\newblock Recurrence versus transience for weight-dependent random connection
  models.
\newblock {\em Electron. J. Probab.}, 27:Paper No. 60, 31, 2022.

\bibitem{glm}
P.~Gracar, L.~L\"{u}chtrath, and P.~M\"{o}rters.
\newblock Percolation phase transition in weight-dependent random connection
  models.
\newblock {\em Adv. Appl. Probab.}, 53(4):1090--1114, 2021.

\bibitem{gpp}
L.~Gugelmann, K.~Panagiotou, and U.~Peter.
\newblock Random hyperbolic graphs: degree sequence and clustering.
\newblock In {\em Automata, Languages, and Programming, {\rm volume 7392 of}
  Lecture Notes in Computer Science}, pages 573--585. Springer, 2012.

\bibitem{hirsch:juhasz:2023}
C.~Hirsch and P.~Juhasz.
\newblock On the topology of higher-order age-dependent random connection
  models.
\newblock {\em Method. Comput. Appl. Probab.}, 2025, forthcoming.

\bibitem{ht24}
C.~Hirsch and T.~Owada.
\newblock Limit theorems under heavy-tailed scenario in the age dependent
  random connection models.
\newblock {\em arXiv preprint arXiv:2409.05226}, 2024.

\bibitem{jacMor1}
E.~Jacob and P.~M\"orters.
\newblock Spatial preferential attachment networks: power laws and clustering
  coefficients.
\newblock {\em Ann. Appl. Probab.}, 25(2):632--662, 2015.

\bibitem{jacMor2}
E.~Jacob and P.~M{\"o}rters.
\newblock Robustness of scale-free spatial networks.
\newblock {\em Ann. Probab.}, 45(3):1680--1722, 2017.

\bibitem{komjathy}
J.~Komj\'{a}thy and B.~Lodewijks.
\newblock Explosion in weighted hyperbolic random graphs and geometric
  inhomogeneous random graphs.
\newblock {\em Stochastic Process. Appl.}, 130(3):1309--1367, 2020.

\bibitem{LrP17}
R.~Lachi\`{e}ze-Rey and G.~Peccati.
\newblock {N}ew {B}erry-{E}sseen bounds for functionals of binomial point
  processes.
\newblock {\em Annals of Applied Probability}, 27:1992--2031, 2017.

\bibitem{y1}
R.~Lachi\`eze-Rey, M.~Schulte, and J.~E. Yukich.
\newblock Normal approximation for stabilizing functionals.
\newblock {\em Ann. Appl. Probab.}, 29(2):931--993, 2019.

\bibitem{mehler}
G.~Last, G.~Peccati, and M.~Schulte.
\newblock Normal approximation on {P}oisson spaces: {M}ehler's formula, second
  order {P}oincar\'{e} inequalities and stabilization.
\newblock {\em Probab. Theory Related Fields}, 165(3-4):667--723, 2016.

\bibitem{poisBook}
G.~Last and M.~D. Penrose.
\newblock {\em Lectures on the Poisson Process}.
\newblock Cambridge University Press, Cambridge, 2016.

\bibitem{newman:1980}
C.~M. Newman.
\newblock Normal fluctuations and the {FKG} inequalities.
\newblock {\em Commun. Math. Phys.}, 74:119--128, 1980.

\bibitem{NorrosReittu}
I.~Norros and H.~Reittu.
\newblock On a conditionally {P}oissonian graph process.
\newblock {\em Adv. Appl. Prob.}, 38(1):59--75, 2006.

\bibitem{oliveira:2012}
P.~E. Oliveira.
\newblock {\em Asymptotics for Associated Random Variables}.
\newblock Springer, Berlin, Heidelberg, 2012.

\bibitem{penrose}
M.~D. Penrose.
\newblock {\em Random Geometric Graphs}.
\newblock Oxford University Press, Oxford, 2003.

\bibitem{yukCLT}
M.~D. Penrose and J.~E. Yukich.
\newblock Central limit theorems for some graphs in computational geometry.
\newblock {\em Ann. Appl. Probab.}, 11(4):1005--1041, 2001.

\bibitem{ustat}
M.~Reitzner and M.~Schulte.
\newblock Central limit theorems for {$U$}-statistics of {P}oisson point
  processes.
\newblock {\em Ann. Probab.}, 41(6):3879--3909, 2013.

\bibitem{y2}
M.~Schulte and J.~E. Yukich.
\newblock Multivariate second order {P}oincar\'e{} inequalities for {P}oisson
  functionals.
\newblock {\em Electron. J. Probab.}, 24, 2019.

\bibitem{shiryaev:1996}
A.~N. Shiryaev.
\newblock {\em Probability, {\rm 2nd edition}}.
\newblock Springer, New York, 1996.

\bibitem{trauth2}
T.~Trauthwein.
\newblock Multivariate second-order $p$-{P}oincar{\'e} inequalities.
\newblock {\em arXiv preprint arXiv:2409.02843}, 2024.

\bibitem{trauth}
T.~Trauthwein.
\newblock Quantitative {CLT}s on the {P}oisson space via {S}korohod estimates
  and $p $-{P}oincar{\'e} inequalities.
\newblock {\em To appear in Ann. Appl. Prob.}, 2025.

\bibitem{book-hof}
R.~van~der Hofstad.
\newblock {\em {R}andom {G}raphs and {C}omplex {N}etworks, {V}olume 2}.
\newblock Cambridge University Press, 2024.

\bibitem{pim}
R.~van~der Hofstad, P.~van~der Hoorn, and N.~Maitra.
\newblock Scaling of the clustering function in spatial inhomogeneous random
  graphs.
\newblock {\em J. Stat. Phys.}, 190(6):Paper No. 110, 43, 2023.

\bibitem{Yu15}
J.~E. Yukich.
\newblock {S}urface order scaling in stochastic geometry.
\newblock {\em Ann. Appl. Prob}, 25:177--210, 2015.

\end{thebibliography}
\end{document}